\documentclass[reqno,11pt]{amsart}

\usepackage[left=3cm,right=3cm,top=3cm,bottom=3cm,footskip=1cm]{geometry}
\usepackage{graphicx}%
\usepackage{multirow}%
\usepackage{amsmath,amssymb,amsfonts}%
\usepackage{amsthm}%
\usepackage{mathrsfs}%
\usepackage{xcolor}%
\usepackage{listings}%
\usepackage{tikz}
\usepackage{hyperref}
\hypersetup{
    colorlinks=true,
    linkcolor=blue,
    filecolor=magenta,      
    urlcolor=cyan,
    citecolor=blue
    }
\usepackage{enumerate}

\newcommand{\TT}{\mathbb{T}}  
\newcommand{\RR}{\mathbb{R}}  
\newcommand{\NN}{\mathbb{N}}  
\newcommand{\ZZ}{\mathbb{Z}}  
\newcommand{\dd}{\mathrm{d}}  
\newcommand{\EE}{\mathbb{E}}  

\newcommand{\TTint}{\int_\mathbb{T}} 
\newcommand{\RRint}{\int_{\mathbb{R}}} 
\newcommand{\spt}{\mathrm{spt}} 

\newcommand{\coup}{\gamma} 
\newcommand{\Parti}{P_{\coup}} 
\newcommand{\alphaliq}{\alpha_{+}} 
\newcommand{\alphavap}{\alpha_{-}} 
\newcommand{\rholiq}{\rho_{+}} 
\newcommand{\rhovap}{\rho_{-}} 
\newcommand{\Cw}{C_{\mathrm{w}}} 
\newcommand{\Cell}{C_{\mathrm{ell}}} 
\newcommand{\Csob}{C_{\mathrm{Sob}}} 
\newcommand{\Kparti}{K^0_{\Parti}} 
\newcommand{\Ku}{K^0_{u}} 
\newcommand{\Kd}{K^0_{d}} 
\newcommand{\tildeT}{\tilde{T}_0} 
\newcommand{\tildeTT}{\tilde{T}} 
\newcommand{\prim}{{\partial_x^{-1}}} 
\newcommand{\othet}{\overline{\Theta}} 

\newcommand{\weakstar}{\overset{\ast}{\rightharpoonup}} 
\newcommand{\weak}{\rightharpoonup} 

\newcommand{\curlyM}{\mathcal{M}} 
\newcommand{\curlyD}{\mathcal{D}} 
\newcommand{\curlyL}{\mathcal{L}} 
\newcommand{\curlyN}{\mathcal{N}} 

\newtheorem{theorem}{Theorem}[section]
\newtheorem{proposition}[theorem]{Proposition}%
\newtheorem{lemma}[theorem]{Lemma}

\theoremstyle{remark}%
\newtheorem{remark}[theorem]{Remark}%

\theoremstyle{definition}%
\newtheorem{definition}[theorem]{Definition}%

\numberwithin{equation}{section}
\begin{document}

\title{Mathematical Justification of a Baer--Nunziato Model for a Compressible Viscous Fluid with Phase Transition}

\author[C. Rohde]{Christian Rohde}
\address[Christian Rohde]{Institute of Applied Analysis and Numerical Simulation, University of Stuttgart, 70569 Stuttgart, Germany}
\email{christian.rohde@mathematik.uni-stuttgart.de}

\author[F. Wendt]{Florian Wendt*}
\address[Florian Wendt]{Institute of Applied Analysis and Numerical Simulation, University of Stuttgart, 70569 Stuttgart, Germany}
\email{florian.wendt@mathematik.uni-stuttgart.de}
\thanks{* Corresponding author}



\date{}

\dedicatory{}

\begin{abstract}

In this work, we justify a Baer--Nunziato system including appropriate closure terms as the macroscopic description of a compressible viscous fluid that can occur in a liquid or a vapor phase in the isothermal framework.
As a mathematical model for the two-phase fluid on the detailed scale we chose a non-local version of the Navier--Stokes--Korteweg equations in the one-dimensional and periodic setting.
Our justification relies on anticipating the macroscopic description of the two-phase fluid as the limit system for a sequence of solutions with highly oscillating initial densities.
Interpreting the density as a parametrized measure, we extract a limit system consisting of a kinetic equation for the parametrized measure and a momentum equation for the velocity.
Under the assumption that the initial density distributions converge in the limit to a convex combination of Dirac-measures, we show by a uniqueness result that the parametrized measure also has to be a convex combination of Dirac-measures and, that the limit system reduces to the Baer--Nunziato system.\\
This work extends existing results concerning the justification of Baer--Nunziato models as the macroscopic description of multi-fluid models in the sense, that we allow for phase transition effects on the detailed scale.
This work also includes a new global-in-time well-posedness result for the Cauchy problem of the non-local Navier--Stokes--Korteweg equations.

\end{abstract}

\maketitle


\section{Introduction}
We consider a homogeneous compressible viscous fluid that can occur in a liquid and a vapor phase.
A widely accepted mathematical description of such a two-phase fluid is given by the Navier--Stokes--Korteweg (NSK) equations (\cite{McFadden}).
These equations model the two-phase fluid with a diffuse interface, i.e. the interface between the liquid and the vapor is assumed to have small but positive Lebesgue-measure and the density varies rapidly but smoothly over the interface.
Let us assume that we are given some positive time $T>0$.
In the isothermal and periodic framework, the NSK model describes then the dynamics of the two-phase fluid via the fluid's density $\rho\colon[0,T) \times \TT \to \RR_{>0}$ and the fluid's velocity $u \colon [0,T) \times \TT \to \RR$ that obey the system of equations
\begin{alignat}{2}\label{NSK}
    \left\{
    \begin{aligned}
        \partial_t \rho + \partial_x(\rho u) &= 0 
        \quad &&\text{in } (0,T)\times \TT,\\
        \partial_t(\rho u ) + \partial_x(\rho u^2)
        + \partial_xP(\rho) - \mu\partial_{xx}u + \kappa \rho\partial_{xxx} \rho &= 0
        \quad &&\text{in } (0,T)\times \TT,
    \end{aligned}
    \right.
\end{alignat}
with initial conditions
\begin{align}\label{NSK Elliptic initial condition}
    \rho(0,\cdot) = \rho_0, \quad u(0,\cdot) = u_0 \quad \text{in } \TT.
\end{align}
Here, $\mu>0$ denotes the constant viscosity coefficient and $\kappa>0$ denotes the constant capillarity coefficient. 
In order to account for phase transition effects, the equation of state $P \colon[0,\infty) \to [0,\infty)$ is assumed to be of Van-der-Waals type.
More precisely, we shall assume that there exist some constants $0<B_1<B_2<\infty$, such that $P$ is monotonically increasing on $[0,B_1] \cup [B_2,\infty)$ and monotonically decreasing on $[B_1,B_2]$.
Accordingly, we call then the fluid's state liquid (spinodal, vapor), if the fluid's density satisfies $\rho \in [0,B_1)$ ($[B_1,B_2), [B_2,\infty)$).
An illustration of a pressure function of Van-der-Waals type is given in Figure~\ref{Pressure_Figure} in Section~\ref{Main Result}.\\
To find effective equations for the NSK model $\eqref{NSK}$, we start from a sequence of initial data $(\rho_n^0,u_n^0)_{n\in\NN}$, where $n \in \NN$ should display the number of phase transitions that we have initially.
Due to the high number of phase transitions, we expect the initial density sequence $(\rho_n^0)_{n\in\NN}$ to be highly oscillating between the vapor density and the liquid density.
As a simplification of the problem, we assume that such oscillations do not occur for the velocity sequence.
After constructing an appropriate corresponding sequence of solutions $(\rho_n,u_n)_{n\in\NN}$, it is then reasonable to assume that the macroscopic equations are found in the limit $n \to \infty$ (i.e. in the limit where the number of phase transitions tends to infinity).
Thus, the derivation of effective equations reduces to study the propagation of initial density oscillations for system $\eqref{NSK}$.
One method to analyze the propagation of initial density oscillations relies on interpreting the density as a parametrized measure (\cite{Serre,Weinan,HillProp}).
However, system $\eqref{NSK}$ seems hardly accessible for such an investigation due to the capillarity term $\kappa \rho \partial_{xxx}\rho$.
To overcome this remedy, we choose a non-local approximation of $\eqref{NSK}$, that was proposed in \cite{Rohde2}.
In this system, that we call the non-local NSK system from now on, the capillarity term $\kappa\rho\partial_{xxx}\rho$ is substituted by a term of lower order by introducing an additional unknown that satisfies an elliptic equation.
More precisely, in the non-local NSK model, the dynamics of the two-phase fluid are described by the fluid's density $\rho\colon[0,T)\times \TT \to \RR_{\geq0}$, the fluid's velocity $u\colon[0,T)\times \TT \to \RR$ and the order parameter $c\colon[0,T) \times \TT \to \TT$ that satisfy the system of equations
\begin{alignat}{2}\label{NSK Elliptic 1}
    \left\{
    \begin{aligned}
        \partial_t \rho + \partial_x(\rho u) &= 0
        \quad &&\text{in } (0,T)\times \TT,\\
        \partial_t(\rho u) + \partial_x(\rho u^2) + \partial_xP(\rho) -\mu\partial_{xx}u - \coup \rho \partial_x(c-\rho) &= 0 
        \quad &&\text{in } (0,T) \times \TT,\\
        -\kappa\partial_{xx}c + \coup (c-\rho) &= 0 
        \quad &&\text{in } (0,T) \times \TT,
    \end{aligned}
    \right.
\end{alignat}
with initial conditions $\eqref{NSK Elliptic initial condition}$.
Here, the quantities $\mu,\kappa,P$ are defined as before and $\coup>0$ denotes the constant coupling coefficient.
Formally, we notice that for $\coup \to \infty$ we recover the compressible NSK system $\eqref{NSK}$.
At this point we refer to \cite{Neusser} for a numerical solution of $\eqref{NSK Elliptic 1}$ and to \cite{Giesselmann} for a rigorous mathematical result concerning the convergence of $\eqref{NSK Elliptic 1}$ to $\eqref{NSK}$ via a relative entropy approach.\\
The momentum equation $\eqref{NSK Elliptic 1}_2$ can be rewritten as
\begin{align}\label{rewritten momentum}
    \partial_t(\rho u)
    +\partial_x(\rho u^2)
    + \partial_x\Parti(\rho)
    - \mu \partial_{xx}u 
    - \coup \rho \partial_x c &= 0,
\end{align}
with an artificial pressure function
\begin{align}\label{artificial pressure}
    \Parti(r) := P(r) + \frac{\coup}{2} r^2
\end{align}
being monotone for a pressure function $P$ of Van-der-Waals type provided $\coup$ is chosen large enough.
In this paper, we justify the following system as a macroscopic description for a compressible liquid-vapor flow that is modeled with the non-local NSK equations on the detailed scale:
\begin{alignat}{2}\label{BN-System to justify}
    \left\{
    \begin{aligned}
        \partial_t \alphaliq + u \partial_x \alphaliq 
        &= 
        \frac{\alphaliq \alphavap}{\mu} \Bigl(P(\rholiq)- P(\rhovap) + \frac{\gamma}{2}(\rholiq^2 - \rhovap^2)\Bigr) \quad &&\text{in } (0,T) \times \TT,\\
        \partial_t \alphavap + u \partial_x \alphavap
        &=
        \frac{\alphavap \alphaliq}{\mu} \Bigl(P(\rhovap) - P(\rholiq) + \frac{\gamma}{2}(\rhovap^2 - \rholiq^2)\Bigr) \quad &&\text{in } (0,T) \times \TT,\\
        \partial_t \rholiq + \partial_x(\rholiq u)
        &=
        \frac{\rholiq\alphavap}{\mu} \Bigl(P(\rhovap)-P(\rholiq) + \frac{\gamma}{2}(\rhovap^2 - \rholiq^2)\Bigr) \quad &&\text{in } (0,T) \times \TT,\\
         \partial_t \rhovap + \partial_x(\rhovap u)
        &=
        \frac{\rhovap\alphaliq}{\mu} \Bigl(P(\rholiq)-P(\rhovap) + \frac{\gamma}{2}(\rholiq^2 - \rhovap^2)\Bigr) \quad &&\text{in } (0,T) \times \TT,\\
        \partial_t(\rho u) 
        +
        \partial_x(\rho u^2)
        &=
        \mu \partial_{xx}u
        -
        \overline{P}
        +
        \coup \rho \partial_x c \quad &&\text{in } (0,T) \times \TT,\\
        -\kappa\partial_{xx}c + \coup c &= \coup \rho \quad &&\text{in } (0,T) \times \TT,
    \end{aligned}
    \right.
\end{alignat}
where
\begin{align*}
    \rho = \alphaliq \rholiq + \alphavap \rhovap,
    \quad
    \overline{P} = \alphaliq P(\rholiq) + \alphavap P(\rhovap) + \frac{\gamma}{2} ( \alpha_+\rholiq^2 + \alpha_-\rhovap^2 ).
\end{align*}
Here, $\alpha_+, \alpha_-, \rho_+, \rho_- \colon [0,T) \times \TT \to \RR$ denote the volume fraction of the liquid phase, the volume fraction of the vapor phase, the partial density of the liquid phase, the partial density of the vapor phase, respectively.
The effective model $\eqref{BN-System to justify}$ falls into the class of one-velocity Baer--Nunziato (BN) multi-fluid models.
For a detailed background on this class of multi-fluid models see e.g. \cite{BAER1986861, SaurelGodunov, MurroneFive, Hantke}.
For a discussion and interpretation of the macroscopic model $\eqref{BN-System to justify}$, we refer to the end of Section~\ref{Main Result}.
Our justification follows the homogenization methodology in \cite{Hillairet_New_Physical}, where the authors justify a one-velocity multi-fluid BN model for the compressible Navier-Stokes equations with density-dependent viscosity in the one-dimensional periodic framework.
There, the authors analyzed the propagation of initial density oscillations by interpreting the density as a parameterized measure and obtained effective equations in the form of a kinetic equation for the parametrized measure and a momentum equation for the velocity.
After characterizing the parameterized measure as a convex combination of Dirac-measures these equations then reduce to a one-velocity multi-fluid BN system.
Hence, the key point to justify such a multi-fluid model is to prove the propagation of the special structure for the parametrized measure in time.
This methodology was also used in \cite{HillProp,Huang,Hill_Note} to justify rigorously a one-velocity BN model for the compressible isentropic Navier--Stokes equations with constant viscosity in the $3$-D framework.
See also \cite{Plotnikov} for the derivation of the kinetic equation in terms of a probability density function without the characterization of the parametrized measure.
At this point, we also refer to \cite{Lagoutiere}, where the authors extended the methodology for a two-component flow with two different pressure laws for each phase and to \cite{LagoutiereSemi} for a semi-discrete approach and a numerical investigation of the problem setting in \cite{Lagoutiere}.
However, all these results do not incorporate phase transition phenomena in the description of the fluid on the detailed scale.
Thus, this work can be seen as a generalization of the results in \cite{Hillairet_New_Physical} to a setting where phase transition effects are included on the detailed scale.
Mathematically, this generalization requires to deal with a non-monotone pressure law and an additional term in the momentum equation.
The difficulties resulting from the additional term in the momentum equation are twofold.
On the one hand, the a-priori estimates in \cite{Hillairet_New_Physical} rely on the momentum equation, so that it is not clear, whether these carry over to the non-local NSK system.
On the other hand, we have to prove convergence for the additional term in the homogenization process.
We overcome the non-monotonicity of the pressure function by rewriting the momentum equation in terms of the artificial pressure function $\Parti$ (see $\eqref{rewritten momentum}$) that we assume to be monotone (c.f. Section~\ref{Main Result}). 
As a novelty, we verify a-priori estimates corresponding to them in \cite{Hillairet_New_Physical} for the non-local NSK system.
In particular an entropy inequality due to Bresch and Desjardin (c.f. Section~\ref{Construction of Global Strong Solutions}) is derived. 
Then, we solve the convergence issue by using a compensated compactness lemma from \cite{Hillairet_New_Physical}. 
As we shall see, this lemma can be applied here, since the a-priori estimates for the order parameter are strong enough.
Our main tool to control the order parameter is an elliptic regularity estimate. 
\\
This paper is organized as follows. 
In Section~\ref{Main Result}, we precisely formulate the assumptions that we impose on the pressure function $P$ and the coupling parameter $\coup$, and state our main results Theorem~\ref{Thm:Global Existence and Uniqueness} and Theorem~\ref{Thm:Main Result}. 
These are given by a global-in-time well-posedness result (Theorem~\ref{Thm:Global Existence and Uniqueness}) for the Cauchy problem $\eqref{NSK Elliptic 1}, \eqref{NSK Elliptic initial condition}$ and a homogenization result (Theorem~\ref{Thm:Main Result}).
In Section~\ref{Construction of Global Strong Solutions}, we give a full proof of the global-in-time well-posedness result. By providing the proof of this result, we also recover crucial a-priori bounds that are used in Section~\ref{Refined A-Priori Estimates}.
In Section~\ref{Refined A-Priori Estimates}, we collect refined a-priori estimates concerning the effective viscous flux that we need for the homogenization procedure.
In Section~\ref{Homogenization}, we perform the homogenization procedure following the methodology in \cite{Hillairet_New_Physical} and conclude the proof of the homogenization result.
In Section~\ref{Conclusions} we give some conclusions.
In the Appendix, we provide some technical results on the existence of solutions to a specific BN system and a uniqueness result for measure-valued solutions to a kinetic equation.


\section*{Notations}\label{Notation}
We denote the one-dimensional torus of period 1 by $\TT$, i.e. $\TT = \RR / \ZZ$. 
For a domain $D$ and $k \in \NN \cup \{ \infty \}$ we denote by $C^k(D)$ the space of $k$-times continuously differentiable functions and by $C^k_c(D)$ the set all functions in $C^k(D)$ having compact support in $D$.
Also, we write in the context of distributions $\curlyD(D):=C^\infty_c(D)$ for the test function space and $\curlyD^\prime(D)$ for the space of distributions defined on $\curlyD(D)$.
In particular, $C^k(\TT)$ denotes the space of all $k$-times continuously differentiable functions on $\TT$ that can be identified with the space of all $k$-times continuously differentiable 1-periodic functions defined on $\RR$. 
We denote for a bounded function $f$ defined on $D$ its supremum-norm on $D$ as $||f||_{L^\infty(D)}$ and the closure of $C^\infty_c(D)$ under the supremum-norm as $C^0_0(D)$.
Since $\TT$ is compact, we have $C^k_c(\TT) = C^k(\TT)$.
For $p \in [1,\infty]$, we denote the Lebesgue space on $D$ as $L^p(D)$ and the $L^p(D)$-norm as $||\cdot||_{L^p(D)}$. 
Thereby, we denote by $W^{k,p}(D)$ the Sobolev space of order $k$ on $D$ and, if $p=2$, we shortly write $H^k(D)$.
We denote the norm on $W^{k,p}(D)$ as $||\cdot||_{W^{k,p}(D)}$.
Specifically for $D=\TT$, we have
\begin{align*}
    ||f||_{L^p(\TT)} = \biggl(\int_{\TT} |f(x)|^p \:\dd x\biggr)^{\frac{1}{p}} =\biggl( \int_0^1 |f(x)|^p \: \dd x \biggr)^{\frac{1}{p}}
    \quad \text{for } f \in L^p(\TT),\: p \in[1,\infty),
\end{align*}
and we define the mean operator on $L^1(\TT)$ as
\begin{align*}
    \EE \colon L^1(\TT) \to \RR, \quad f \mapsto \EE(f) := \TTint f(x) \: \dd x.
\end{align*}
Then, we denote the $k$-th Sobolev space of mean free functions on $\TT$ as
\begin{align*}
    \dot{H}^k(\TT):= \{ f \in H^k(\TT) \mid \EE(f) = 0\} \subseteq H^k(\TT),
\end{align*}
and, for $k\geq 1$, we denote its dual space as
\begin{align*}
    \dot{H}^{-k}(\TT) := (\dot{H}^k(\TT))^\star.
\end{align*}
It is easy to see that there exists a mean free primitive operator on $\dot{H}^k(\TT)$ for $k \geq 0$ and we denote this operator as
\begin{align*}
    \prim\colon \dot{H}^k(\TT) \to \dot{H}^{k+1}(\TT), \quad f \mapsto \prim f,
\end{align*}
where $\partial_x \prim f = f$. 
This operator straightforwardly extends to negative Sobolev spaces, i.e. to exponents $k <0$, via
\begin{align*}
    \langle \prim f, \phi \rangle := \langle f, \prim \phi \rangle
\end{align*}
for any $f \in \dot{H}^{k}(\TT)$ and any $\phi \in \dot{H}^{k+1}(\TT)$.
For some Banach space $B$ and some time $T>0$, we denote Bochner space of time dependent functions with values in $B$ as $L^p(0,T;B)$ and the Bochner norm on $L^p(0,T;X)$ as $||\cdot||_{L^p(0,T;B)}$. 
Also, for some interval $I \subseteq \RR$, we denote by $C(I;B)$ the space of continuous functions from $I$ into the Banach space $B$.
The space of weakly continuous functions from $I$ into $B$, i.e. the space of all functions from $I$ into $B$ that are continuous with respect to the weak topology on $B$ is denoted by $C_{\mathrm{w}}(I;B)$. 
Note that $L \in \Cw(I;B)$ if and only if for any $\phi \in B^\star$, where $B^\star$ denotes the dual space of $B$, the map
\begin{align*}
    I \ni t \mapsto \langle \phi, L_t \rangle_{B^\star,B} \in \RR
\end{align*}
is a continuous function.
Finally, for a locally compact Hausdorff space $X$, we denote the space of signed finite Radon measures on $X$ as $\curlyM(X)$. This space can be naturally identified with the dual space of $C^0_0(X)$. For some $\mu \in \curlyM(X)$, we denote its total variation norm as 
\begin{align*}
    ||\mu||_{\curlyM(X)} := \sup\limits_{\phi \in C^0_0(X),\: |\phi|\leq 1} \langle \mu, \phi \rangle.
\end{align*}
The space of positive finite Radon measures on $X$ is denoted by $\curlyM^+(X)$, the space of compactly supported finite Radon measures on $X$ is defined by $\curlyM_c(X)$ and the space of positive compactly supported finite Radon measures on $X$ is denoted by $\curlyM^+_c(X)$.


\section{The Main Result}\label{Main Result}

The first result of this paper is the global-in-time existence of strong solutions for the Cauchy problem $\eqref{NSK Elliptic 1}, \eqref{NSK Elliptic initial condition}$.
This result requires the artificial pressure function $\Parti$ (see $\eqref{artificial pressure}$) to be monotone (c.f. Section~\ref{Construction of Global Strong Solutions}). 
This will impose a certain relation between the pressure function $P$ and the coupling parameter $\coup$.
In the following definition, we make precise what kind of assumptions on $P$ and $\coup$ we need in this paper.
See \cite{Wolff} for a corresponding definition.

\begin{definition}[Admissible Pressure Function]\label{Defi:Admissible Pressure Function}
    Let $P\colon [0,\infty) \to [0,\infty)$ and $\coup>0$ satisfy
    \begin{enumerate}[1.]
        \item \label{condition_one_admissible_pressure}  $P \in C^1([0,\infty))$, $P(0) = 0$,
        \item \label{condition_two_admissible_pressure} there exist two constants $\beta\in [2,\infty)$ and $a \in (0,\infty)$, such that $\lim\limits_{r \to \infty} \frac{P^\prime(r)}{r^{\beta - 1}} = a>0$,
        \item the artificial pressure function $\Parti$ defined in $\eqref{artificial pressure}$ is monotonically increasing on $[0,\infty)$.
    \end{enumerate}
    Then we call the pair $(P,\gamma)$ \textit{admissible}.
\end{definition}


\begin{remark}
    The Van-der-Waals pressure law is a widely accepted equation of state for modeling two-phase phenomena of a homogeneous fluid. 
    At a constant reference temperature $T_\star>0$, it is given by
    \begin{align*}
        P_{\mathrm{VdW}}\colon [0,B) \to [0,\infty), \quad \rho \mapsto
        P_{\mathrm{VdW}}(\rho) := \frac{RT_\star \rho}{B-\rho} - A \rho^2,
    \end{align*}
    where $A,B,R>0$ are material coefficients. 
    For the reference temperature $T_\star>0$ small enough, more precisely
    \begin{align*}
        T_\star <\frac{AB^2}{R},
    \end{align*}
    the pressure function $P_{\mathrm{VdW}}$ admits a spinodal region in the sense of the preceding discussion.
    The derivative of $P_{\mathrm{VdW}}$ is given by
    \begin{align*}
        P_{VdW}^\prime(r) = \frac{BRT_\star}{(B-r)^2} - 2Ar.
    \end{align*}
    In particular, we have that, up to a change of coordinate, $(P_{\mathrm{VdW}},\coup)$ is admissible for 
    \begin{align*}
        \coup\geq 2A.
    \end{align*}
\end{remark}
Let us assume that we are given some admissible pair $(P,\coup)$.
Then we associate to $P\colon[0,\infty) \to [0,\infty)$ some pressure potential $W\colon[0,\infty) \to [0,\infty)$ through the relation
\begin{align}\label{relation pressure potential}
    P^\prime(r) = W^{\prime\prime}(r) r.
\end{align}
We notice that monotonically increasing/decreasing regions of $P$ correspond to convex/concave regions of $W$.
\begin{remark}
    Condition \ref{condition_two_admissible_pressure} in Definition~\ref{Defi:Admissible Pressure Function} implies that there exists some positive constant $C_\rho$, such that
    \begin{align}\label{bound for rbeta}
        r^2 \leq C_\rho + C_\rho W(r) \quad \forall \, r \in [0,\infty).
    \end{align}
\end{remark}

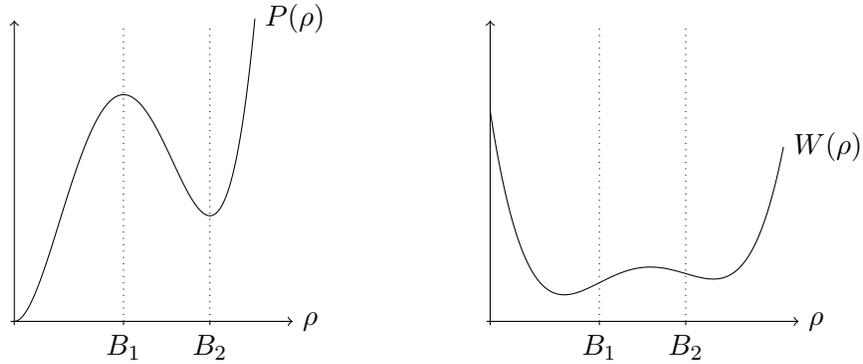
\begin{figure}[h]
\centering
\begin{tikzpicture}
  \draw[->] (-0.05,0) -- (3.7,0) node[right] {$\rho$};
  \draw[->] (0,-0.05) -- (0,4.0);
  \draw[dotted] (1.452,0) -- (1.452,3.9);
  \draw[dotted] (2.6,0) -- (2.6,3.9);
  \draw[color=black, domain=0.001:3.2, samples=200] plot (\x,{0.7*(1.125*\x^4-6.08*\x^3+8.5*\x^2)}) node[right] {$P(\rho)$};
  \draw (1.452,0) -- (1.452,-0.05) node[below] {$B_1$};
  \draw (2.6,0) -- (2.6,-0.05) node[below] {$B_2$};
  \draw[->] [xshift = 180.0] (-0.05,0) -- (4.1,0) node[right] {$\rho$};
  \draw[->] [xshift = 180.0] (0,-0.05) -- (0,4.0);
  \draw[color=black, domain=0:3.9, samples=200, xshift = 180.0] plot (\x,{0.7*(0.375*\x^4-3.04*\x^3+8.5*\x^2-9.33*\x+4)}) node[right] {$W(\rho)$};
  \draw[dotted, xshift = 180.0] (1.452,0) -- (1.452,3.9);
  \draw[dotted, xshift = 180.0] (2.6,0) -- (2.6,3.9);
  \draw[xshift = 180.0] (1.452,0) -- (1.452,-0.05) node[below] {$B_1$};
  \draw[xshift = 180.0] (2.6,0) -- (2.6,-0.05) node[below] {$B_2$};
\end{tikzpicture}
\caption{Left: Example for a pressure function $P(\rho)$ of Van-der-Waals type. Right: A corresponding pressure potential $W(\rho)$.}\label{Pressure_Figure}
\end{figure}

Our first result then reads as follows:

\begin{theorem}\label{Thm:Global Existence and Uniqueness}
    Let $\kappa,\mu,\gamma,M_0>0$ and assume that $(P,\gamma)$ is admissible.
    Let $u_0, \rho_0 \in H^1(\TT)$ with $0 < M_0^{-1} \leq \rho_0(x) \leq M_0$ for all $x \in \TT$.\\
    Then there exist unique functions
    \begin{alignat*}{2}
        &\rho \in C([0,\infty);H^1(\TT)), \quad \rho>0 \quad \text{on } [0,\infty)\times\TT, \quad &&\partial_t \rho \in C([0,\infty);L^2(\TT)),\\ 
        &u \in  C([0,\infty);H^1(\TT)) \cap L^2_{\mathrm{loc}}(0,\infty;H^2(\TT)),\quad &&\partial_t u \in L^2_{\mathrm{loc}}(0,\infty;L^2(\TT)),\\
        &c \in C([0,\infty);H^3(\TT)),
    \end{alignat*}
    that satisfy $\eqref{NSK Elliptic 1}$ a.e. on $(0,\infty)\times\TT$, the initial conditions $\eqref{NSK Elliptic initial condition}$ a.e. on $\TT$ and
    \begin{align*}
        -\kappa \partial_{xx}c(t) + \coup c(t) = \coup \rho(t) \quad \text{a.e. on } \TT
    \end{align*}
    for any $t \in [0,\infty)$.
\end{theorem}


The global-in-time result Theorem~\ref{Thm:Global Existence and Uniqueness} is then used to formulate our second result, which is the main result of this paper.
Before we state this result, let us recall how a $L^\infty(\TT)$-function gives rise to a positive finite Radon measure on $\TT \times \RR$. Let $f \in L^\infty(\TT)$. Then we can define
\begin{align*}
    \Theta\colon C^0_0(\TT \times \RR) \to \RR,
    \quad 
    b \mapsto \langle \Theta, b \rangle := \TTint b(x,f(x)) \:\dd x.
\end{align*}
It is an easy matter to check that this defines a positive linear functional on $C^0_0(\TT \times \RR)$, so that indeed $\Theta \in \curlyM^+(\TT \times \RR)$. 
If we assume furthermore that $f$ is bounded from above and below, i.e. 
\begin{align*}
    C_1 \leq f(x)  \leq C_2 \quad \text{for a.e. } x \in \TT,
\end{align*}
where $C_1<C_2$ are two constants, then we obtain that $\Theta$ has compact support, more precisely
\begin{align*}
    \spt \, (\Theta) \subseteq \TT \times [C_1,C_2].
\end{align*}
This implies $\Theta \in \curlyM_c^+(\TT \times \RR)$. We will use these notions in the following main result, which justifies a one-velocity BN system as a macroscopic description for a compressible liquid-vapor flow that is described on the detailed scale by the non-local NSK system \eqref{NSK Elliptic 1}.

\begin{theorem}[Main Result]\label{Thm:Main Result}
    Let $\kappa,\mu,\gamma,M_0>0$ and assume that $(P,\gamma)$ is admissible.
    Suppose that $\rho_n^0, u_n^0 \in H^1(\TT)$ satisfy 
    \begin{align*}
         M_0^{-1} \leq \rho_n^0(x) \leq M_0 \quad \forall\, x \in \TT,\quad \forall \,n \in \NN, \qquad
        \sup\limits_{n\in\NN} \bigl\{||u_n^0||_{H^1(\TT)} \bigr\} < \infty,
    \end{align*}
    for some positive constant $M_0$ and, that $u_n^0 \weak u_0$ weakly in $H^1(\TT)$.
    For $n \in \NN$, let $(\rho_n,u_n,c_n)$ denote the global-in-time strong solution of $\eqref{NSK Elliptic 1}, \eqref{NSK Elliptic initial condition}$ with initial data $(\rho_n^0,u_n^0)$ (c.f. Theorem~\ref{Thm:Global Existence and Uniqueness}).
    Assume that there exists $\alphaliq^0,\alphavap^0,\rholiq^0,\rhovap^0 \in L^\infty(\TT)$ with
    \begin{align*}
        0 \leq \alpha_{0\pm} \leq 1, \quad \alpha_{0+} + \alpha_{0-} = 1,\quad
        M_0^{-1} \leq \rho_{0\pm} \leq M_0 \quad \text{a.e. on } \TT,
    \end{align*}
    such that $\Theta_n^0 \in \curlyM_c^+(\TT \times \RR)$, defined through
    \begin{align*}
        \langle \Theta_n^0, b \rangle := \TTint b(x,\rho_n^0(x))\, \dd x \quad  \forall\, b \in C^0_0(\TT \times \RR),
    \end{align*}
    satisfies
    \begin{align}\label{assumption:initial density distribution}
        \langle \Theta_n^0, b \rangle 
        \longrightarrow
        \TTint \alphaliq^0(x) b(\rholiq^0(x)) + \alphavap^0(x) b(\rhovap^0(x)) \, \dd x
        \quad \forall \,b \in C^0_0(\TT \times \RR).
    \end{align}
    Then there exist some time $T_1>0$, $\alphaliq,\alphavap,\rholiq,\rhovap\in L^\infty(0,T_1;L^\infty(\TT)) \cap C([0,T_1];L^1(\TT))$ and $u \in L^\infty(0,T_1;H^1(\TT)) \cap C([0,T_1];C(\TT))$, such that (up to a subsequence) $\Theta_n \in \Cw([0,T_1];\curlyM^+(\TT \times \RR))$, defined for any $t \in [0,T_1]$ through
    \begin{align*}
        \langle \Theta_n(t), b \rangle 
        := \TTint b(x,\rho_n(t,x)) \, \dd x \quad \forall \, b \in C^0_0(\TT \times \RR),
    \end{align*}
    converges in $\Cw([0,T_1];\curlyM^+(\TT \times \RR))$ to $\Theta$, where $\Theta$ is given for any $t \in [0,T_1]$ through
    \begin{align*}
        \langle \Theta(t) , b \rangle 
        = \TTint \alphaliq(t,x) b(x, \rholiq(t,x)) 
        +
        \alphavap(t,x) b(x,\rhovap(t,x)) \,\dd x
        \quad \forall\, b \in C^0_0(\TT \times \RR).
    \end{align*}
    Moreover, we have
    \begin{align*}
        u_n \to u \quad \text{in } C([0,T_1];C(\TT)),
        \quad 
        c_n \weak c \quad \text{in } L^2(0,T_1;H^2(\TT)).
    \end{align*}
    The functions
    $\alpha_+, \alpha_-,\rho_+,\rho_-,u,c$ satisfy the BN system $\eqref{BN-System to justify}$ in $\curlyD^\prime((0,T_1)\times \TT)$ and initial conditions
    \begin{align*}
        \alphaliq(0,\cdot)=\alphaliq^0,
        \quad
        \alphavap(0,\cdot) = \alphavap^0,
        \quad
        \rholiq(0,\cdot) = \rholiq^0,
        \quad
        \rhovap(0,\cdot) = \rhovap^0 \quad
        u(0,\cdot) = u_0
    \end{align*}
    a.e. in $\TT$.
\end{theorem}

To the end of this section let us discuss how our results justify the BN system $\eqref{BN-System to justify}$ as the macroscopic description of a two-phase fluid modeled by the non-local NSK system on the detailed scale.
We start with a sequence of initial data $(\rho_n^0,u_n^0)_{n\in\NN}$ and construct the corresponding sequence of solutions $(\rho_n,u_n,c_n)_{n\in\NN}$ according to Theorem~\ref{Thm:Global Existence and Uniqueness}.
We anticipate to find the macroscopic description of the two-phase fluid in the limit $n \to \infty$, when the initial density sequence highly oscillates between the liquid and the vapor density, for instance we could have
\begin{align*}
    \rho^0_n(x) = \rho^0(nx)
\end{align*}
for some density profile $\rho^0 \in H^1(\TT)$ bounded below by some positive constant.
The continuity assumption on the initial density is physically reasonable here, since we deal with a diffuse interface model.
Since we expect that the initial density sequence is highly oscillating between the liquid and the vapor density, it is reasonable to assume that the corresponding sequence of initial density distributions $(\Theta_n^0)_{n\in\NN}$ concentrates in the limit on two values in the sense of assumption $\eqref{assumption:initial density distribution}$.
Under this assumption, Theorem~\ref{Main Result} tells us then that for a small time, after identifying the density sequence $(\rho_n)_{n\in\NN}$ as a sequence of parametrized measures $(\Theta_n)_{n\in\NN}$, we find a limit velocity $u$ and a limit density distribution $\Theta$, such that $(\Theta_n,u_n)_{n\in\NN}$ converges to $(\Theta,u)$ in an appropriate way and moreover, the structure of the initial density distribution is preserved with functions $\alpha_+,\alpha_-,\rho_+,\rho_-$ that solve the BN system $\eqref{BN-System to justify}$.
Agreeing that $(\Theta,u)$ describes our two-phase fluid on the macroscopic scale, this justifies the BN system $\eqref{BN-System to justify}$ as the macroscopic description.
Since on the detailed scale, the NSK system $\eqref{NSK}$ approximates $\eqref{NSK Elliptic 1}$ for $\coup \to \infty$, it would be interesting, whether one could find a limit system for $\coup \to \infty$ in $\eqref{BN-System to justify}$ and whether the approximation $\coup \to \infty$ commutes with the homogenization limit.
The limit BN system for $\coup\to \infty$ could then be interpreted as a macroscopic model for the NSK system $\eqref{NSK}$.



\section{Construction of Global-in-Time Strong Solutions}\label{Construction of Global Strong Solutions}
In this section, we provide the proof of Theorem~\ref{Thm:Global Existence and Uniqueness}.
The uniqueness of strong solutions for the Cauchy problem $\eqref{NSK Elliptic 1},\eqref{NSK Elliptic initial condition}$ can be proven by classical arguments (see e.g. in \cite{Valli}).
Therefore, we omit the proof for the uniqueness part here and focus on the existence part.
For this part, we follow the lines in \cite{Mellet} and provide a continuation argument for the local-in-time strong solution.
This continuation argument relies on suitable a-priori estimates for the local-in-time strong solution.
In comparison to the compressible Navier--Stokes equations, we have to deal with the additional term $\coup \rho \partial_x c$ in the momentum equation.
Exploiting elliptic regularity results, we are able to control this term in such a way, that the a-priori estimates from \cite{Mellet} carry over.
Our proof relies on the following local-in-time existence result for the Cauchy problem $\eqref{NSK Elliptic 1}, \eqref{NSK Elliptic initial condition}$.

\begin{proposition}\label{prop:local-in-time}
    Let the hypothesis of Theorem~\ref{Thm:Global Existence and Uniqueness} hold true.\\
    Then there exist a time $T_0>0$ that only depends on
    \begin{align*}
        ||u_0||_{H^1(\TT)}, ||\rho_0||_{H^1(\TT)}, M_0, \coup,\kappa, \mu,
    \end{align*}
    and unique functions
    \begin{alignat*}{2}
        &\rho \in C([0,T_0);H^1(\TT)), \quad \rho >0 \quad \text{on } [0,T_0) \times \TT, \quad &&\partial_t \rho \in C([0,T_0);L^2(\TT)),\\
        &u \in C([0,T_0);H^1(\TT)) \times L^2(0,T_0;H^2(\TT)), \quad &&\partial_tu \in L^2(0,T_0;L^2(\TT)),\\
        &c \in C([0,T_0);H^3(\TT))
    \end{alignat*}
    that satisfy $\eqref{NSK Elliptic 1}$ a.e. on $(0,T_0)\times\TT$, the initial conditions $\eqref{NSK Elliptic initial condition}$ a.e. in $\TT$, and
    \begin{align}\label{local-in-time-existence elliptic equation}
        -\kappa \partial_{xx} c(t) + \coup c(t) = \coup \rho(t) \quad \text{a.e. on } \TT
    \end{align}
    for any $t \in [0,T_0)$.
\end{proposition}

The proof of Proposition~\ref{prop:local-in-time} can be obtained via the method of successive approximation as in \cite{Solonnikov} or via a fixed point argument as in \cite{Valli}.
We omit a proof here.

Let us denote by $(\rho,u,c)$ the local-in-time strong solution existing on $[0,T_0)$. 
If we show that
\begin{align}
    &C(T_0)^{-1} \leq \rho(t,x) \leq C(T_0) \quad \forall \, (t,x) \in [0,T_0)\times \TT,\nonumber\\
    &||\rho||_{L^\infty(0,T_0;H^1(\TT))} + ||u||_{L^\infty(0,T_0;H^1(\TT))} \leq C(T_0)\label{bounds:continuation argument}
\end{align}
for some $C(T_0)>0$ only depending on the initial data, a continuation argument yields that the local solution $(\rho,u,c)$ must be global, i.e. $T_0 = \infty$.
The goal for the rest of this section is the verification of the bounds in $\eqref{bounds:continuation argument}$.
Throughout this whole section $C(T_0)>0$ denotes a generic constant that may vary from line to line but only depends on
$
    ||\rho_0||_{H^1(\TT)}, ||u_0||_{H^1(\TT)}, M_0, \mu,\kappa,\coup, T_0.
$
To verify the bounds $\eqref{bounds:continuation argument}$, we exploit the classical energy estimate and an entropy inequality (BD-entropy inequality) that was first derived by the authors in \cite{Bresch_BD_Entropy} for capillary fluids of Korteweg type and then verified by the authors in \cite{Mellet, Mellet_Baro} for the compressible Navier--Stokes system with density dependent viscosity.
Before we come to these inequalities, let us state the following auxiliary lemma that exploits the fact that $c$ satisfies an elliptic equation.

\begin{lemma}\label{lem:elliptic regularity}
    Assume that the hypotheses of Theorem~\ref{Thm:Global Existence and Uniqueness} hold true and let $(\rho,u,c)$ denote the local-in-time strong solution of $\eqref{NSK Elliptic 1}, \eqref{NSK Elliptic initial condition}$ that exists on $[0,T_0)$.\\
    Then there exists a constant $\Cell>0$ that does only depend on $\kappa,\coup$, such that for all $t \in [0,T_0)$, we have the inequality
    \begin{align*}
        ||c(t)||_{H^2(\TT)} \leq \Cell || \rho(t) ||_{L^2(\TT)}.
    \end{align*}
\end{lemma}
\begin{proof}
    For any $t \in [0,T_0)$ we have that $c(t)$ satisfies the elliptic equation $\eqref{local-in-time-existence elliptic equation}$.
    Thus, Lemma~\ref{lem:elliptic regularity} follows from standard interior regularity estimates, see e.g. \cite{evans2010partial}.
\end{proof}

Now we come to the classical energy estimate.

\begin{proposition}[Energy dissipation]\label{Prop:Energy Dissipation}
    Assume that the hypotheses of Theorem~\ref{Thm:Global Existence and Uniqueness} hold true and let $(\rho,u,c)$ denote the local-in-time strong solution of $\eqref{NSK Elliptic 1}, \eqref{NSK Elliptic initial condition}$ that exists on $[0,T_0)$.\\
    Then we have for all $t \in [0,T_0)$ the inequality
    \begin{align*}
        &\TTint 
        \frac{1}{2} \rho u^2 
        +
        W(\rho)
        +
        \frac{\coup}{2} |\rho - c|^2 
        +
        \frac{\kappa}{2} |\partial_x c|^2 \,\dd x
        +
        \int_0^t \TTint \mu |\partial_x u|^2 \, \dd x \, \dd \tau
        \leq
        E_0,
    \end{align*}
    where
    \begin{align*}
        E_0 := \frac{1}{2}M_0||u_0||_{L^2(\TT)}^2 +  \max\limits_{\lambda\in[M_0^{-1},M_0]}\bigl\{|W(\lambda)|\bigr\}
        +  \gamma M_0^2 + \gamma \Cell^2M_0^2 + \frac{\kappa}{2} \Cell^2 M_0^2.
    \end{align*}
\end{proposition}
\begin{proof}
    Since we are in the strong solution framework, the following calculations are justified. First, using the continuity equation $\eqref{NSK Elliptic 1}_1$, we may rewrite the momentum equation $\eqref{NSK Elliptic 1}_1$ as
    \begin{align*}
        \rho (\partial_t u + u\partial_x u)
        +
        P^\prime(\rho) \partial_x \rho
        -
        \mu \partial_{xx} u
        +
        \coup \rho \partial_x(\rho - c)
        = 0.
    \end{align*}
    Multiplying this equation by $u$, and integrating over the spatial domain $\TT$ yields then for almost all $t\in[0,T_0)$ the relation
    \begin{align}\label{Prop:Energy Dissipation Proof 1}
        \TTint \rho(\partial_t u + u \partial_x u)u \, \dd x 
        +
        \TTint P^\prime(\rho) \partial_x \rho u \,\dd x
        -
        \mu \TTint \partial_{xx}u u \,\dd x 
        + 
        \coup \TTint \rho \partial_x(\rho-c) u \,\dd x 
        =
        0.
    \end{align}
    Now we analyze each term in equation $\eqref{Prop:Energy Dissipation Proof 1}$ separately. For the first term, we obtain by virtue of the continuity equation $\eqref{NSK Elliptic 1}_1$ for almost all $t \in [0,\infty)$
    \begin{align}\label{Prop:Energy Dissipation Proof 2}
        \TTint \rho(\partial_t u + u \partial_x u)u \, \dd x 
        =
        \frac{\dd}{\dd t} \TTint \frac{1}{2} \rho |u|^2\,  \dd x.
    \end{align}
    For the second term, we obtain again by virtue of the continuity equation $\eqref{NSK Elliptic 1}_1$ and relation $\eqref{relation pressure potential}$ using integration by parts
    \begin{align}
        \TTint P^\prime(\rho) \partial_x \rho u \, \dd x
        &=
        \TTint W^{\prime\prime}(\rho) (\partial_x \rho) \rho u \, \dd x
        =
        \TTint \partial_x(W^\prime(\rho)) \rho u \,\dd x \nonumber\\
        &=
        -\TTint W^\prime(\rho) \partial_x(\rho u) \, \dd x
        =
        \TTint W^\prime(\rho) \partial_t\rho \, \dd x\nonumber\\
        &=
        \frac{\dd}{\dd t} \TTint W(\rho)\,  \dd x.\label{Prop:Energy Dissipation Proof 3}
    \end{align}
    For the third term in equation $\eqref{Prop:Energy Dissipation Proof 1}$, integration by parts yields
    \begin{align}\label{Prop:Energy Dissipation Proof 4}
        -\mu \TTint u\partial_{xx}u \,  \dd x 
        =
        \mu \TTint |\partial_x u|^2 \, \dd x.
    \end{align}
    For the last term, we use again the continuity equation $\eqref{NSK Elliptic 1}_1$ and integration by parts to obtain
    \begin{align*}
        \TTint \coup \partial_x(\rho - c) \rho u \, \dd x
        =
        \TTint \partial_t \rho \coup (\rho - c) \, \dd x.
    \end{align*}
    Multiplying the elliptic equation $\eqref{NSK Elliptic 1}_3$ by $\partial_t c$ and integrating over the spatial domain $\TT$ yields for almost all $t \in [0,T_0)$
    \begin{align*}
        0
        =
        \TTint \kappa \partial_t c \partial_{xx}c \, \dd x
        +
        \TTint \coup \partial_t c (\rho - c) \, \dd x
        =
        -
        \frac{\dd}{\dd t} \TTint \frac{\kappa}{2} |\partial_x c|^2 \, \dd x
        +
        \TTint \coup \partial_t c (\rho - c) \, \dd x.
    \end{align*}
    Thus,
    \begin{align}\label{Prop:Energy Dissipation Proof 7}
        \TTint \coup \partial_x (\rho - c) \rho u \, \dd x
        &= \frac{\dd}{\dd t} \TTint \frac{\coup}{2} |\rho-c|^2 \dd x + \TTint \coup \partial_t c (\rho - c)\,  \dd x\nonumber\\
        &= \frac{\dd}{\dd t} \TTint \frac{\coup}{2} |\rho - c|^2  + \TTint \frac{\kappa}{2} |\partial_x c|^2\,  \dd x.
    \end{align}
    Together, equations $\eqref{Prop:Energy Dissipation Proof 1}$, $\eqref{Prop:Energy Dissipation Proof 2}$, $\eqref{Prop:Energy Dissipation Proof 3}$, $\eqref{Prop:Energy Dissipation Proof 4}$ and $\eqref{Prop:Energy Dissipation Proof 7}$ yield for almost all $t \in [0,T_0)$
    \begin{align*}
        \frac{\dd}{\dd t} \TTint \rho |u|^2 + W(\rho) + \frac{\coup}{2} |\rho-c|^2 + \frac{\kappa}{2} |\partial_x c|^2 \, \dd x 
        +
        \TTint \mu |\partial_x u|^2 \, \dd x
        =
        0.
    \end{align*}
    After integration in time we obtain for any $t \in [0,T_0)$
    \begin{align*}
        &\TTint \rho |u|^2 + W(\rho) + \frac{\coup}{2}|\rho-c|^2 + \frac{\kappa}{2} |\partial_x c|^2 \, \dd x +
        \int_0^t \TTint \mu |\partial_x u|\, \dd x \, \dd \tau\\
        &\leq
        \TTint \rho_0(x) |u_0(x)|^2 + W(\rho_0(x)) + \frac{\coup}{2}|\rho_0(x) - c(0,x)|^2 + \frac{\kappa}{2} |\partial_x c(0,x)|^2 \, \dd x\\
        &\leq E_0,
    \end{align*}
    where in the last inequality we have used Hölder's and Young's inequality and Lemma~\ref{lem:elliptic regularity} for $t=0$.
\end{proof}

Combining Proposition~\ref{Prop:Energy Dissipation} and relation $\eqref{bound for rbeta}$, we obtain some control on the density.

\begin{proposition}\label{prop:square control on density}
    Assume that the hypotheses of Theorem~\ref{Thm:Global Existence and Uniqueness} hold true and let $(\rho,u,c)$ denote the local-in-time strong solution of $\eqref{NSK Elliptic 1}, \eqref{NSK Elliptic initial condition}$ that exists on $[0,T_0)$.\\ 
    Then there exists some constant $E_1\geq0$ that only depends on $M_0,||u_0||_{L^2(\TT)},\mu,\coup,\kappa$, such that for any $t \in [0,T_0)$, we have
    \begin{align*}
        ||\rho(t)||_{L^2(\TT)}  + ||c(t)||_{W^{1,\infty}(\TT)} + ||c(t)||_{H^2(\TT)}\leq E_1.
    \end{align*}
\end{proposition}
\begin{proof}
    With relation $\eqref{bound for rbeta}$ and Proposition~\ref{Prop:Energy Dissipation} we have
    \begin{align*}
        \TTint \rho^2 \, \dd x\leq C_\rho + C_\rho \TTint W(\rho) \, \dd x \leq C_\rho + C_\rho E_0.
    \end{align*}
    Using the Sobolev embedding $W^{1,2}(\TT) \hookrightarrow L^\infty(\TT)$ and the elliptic estimate provided by Lemma~\ref{lem:elliptic regularity}, we obtain
    \begin{align*}
        ||c(t)||_{W^{1,\infty}(\TT)} 
        \leq C ||c(t)||_{H^2(\TT)}
        \leq C \Cell  ||\rho(t)||_{L^2(\TT)}
        \leq C\Cell \sqrt{C_\rho + C_\rho E_0},
    \end{align*}
    where $C$ denotes the embedding constant of $W^{1,2}(\TT) \hookrightarrow L^\infty(\TT)$. 
    Setting
    \begin{align*}
        E_1 := 3\max\biggl\{\sqrt{ C_\rho + C_\rho E_0 }, C \Cell\sqrt{C_\rho + C_\rho E_0}, \Cell \sqrt{C_\rho + C_\rho E_0} \biggr\}
    \end{align*}
    yields the claim.
\end{proof}

As a next step, we derive the BD-entropy inequality. From this inequality we gain some control on the derivative of the density. Compared to the situation for the compressible Navier--Stokes equations in \cite{Mellet}, we have to treat an additional term involving the order parameter in the momentum equation.
Fortunately, the a-priori estimates for $c$ derived in Proposition~\ref{prop:square control on density} are strong enough to control this additional term. We emphasize that the monotonicity of $\Parti$ is crucial in order to use the following result.

\begin{proposition}\label{prop:BD}
    Assume that the hypotheses of Theorem~\ref{Thm:Global Existence and Uniqueness} hold true and let $(\rho,u,c)$ denote the local-in-time strong solution of $\eqref{NSK Elliptic 1}, \eqref{NSK Elliptic initial condition}$ that exists on $[0,T_0)$.\\
    Then we have for all $t \in [0,T_0)$ the inequality
    \begin{align}\label{prop:BD:eq1}
        &\TTint \frac{1}{2} \rho |u + \partial_x\varphi(\rho)|^2 + W(\rho) + \frac{\coup}{2} |\rho - c|^2 + \frac{\kappa}{2} |\partial_x c|^2  \, \dd x\nonumber\\
        &\qquad+
        \int_0^t \TTint \partial_x(\varphi(\rho)) \partial_x \Parti(\rho) \, \dd x \, \dd t  
        \leq 
        C(T_0),
    \end{align}
    where $\varphi(r):= \int_1^r \frac{\mu}{s^2} \, \dd s$,
    and in particular
    \begin{align}\label{prop:BD:eq2}
        \biggl|\biggl|\partial_x\biggl(\frac{1}{ \sqrt{\rho}}\biggr)\biggr|\biggr|_{L^\infty(0,T_0;L^2(\TT))} \leq C(T_0).
    \end{align}
\end{proposition}
\begin{proof}
    For the proof, we follow the lines in \cite{Mellet}.
    In the proof of Proposition~\ref{Prop:Energy Dissipation}, we have already verified the relation
    \begin{align}\label{prop:BD:proof:eq1}
        \frac{\mathrm{d}}{\mathrm{d}t}
        \TTint \frac{1}{2} \rho |u|^2 + W(\rho) + \frac{\coup}{2} |\rho-c|^2 + \frac{\kappa}{2} |\partial_x c|^2 \, \dd x
        =
        -\mu \TTint |\partial_xu|^2 \, \dd x.
    \end{align}
    We infer with a mollification argument for the density and integration by parts the relation 
    \begin{align*}
        \frac{\mathrm{d}}{\mathrm{d}t} \TTint \rho u \partial_x \varphi(\rho) \, \dd x
        =
        \TTint \partial_t(\rho u) \partial_x \varphi(\rho) \, \dd x
        -
        \TTint \partial_x(\rho u) \varphi^\prime(\rho) \partial_t \rho \, \dd x.
    \end{align*}
    On the one hand, the momentum equation $\eqref{NSK Elliptic 1}_2$ yields then
    \begin{align*}
        \TTint \partial_t(\rho u ) \partial_x \varphi(\rho) \, \dd x
        =
        &\TTint\mu \partial_x\varphi(\rho) \partial_{xx}u 
        +  \coup \rho \partial_x c \partial_x \varphi(\rho)
        - \partial_x(\rho u^2) \partial_x\varphi(\rho)\\
        &\qquad-  \partial_x \Parti(\rho) \partial_x \varphi(\rho) \, \dd x
    \end{align*}
    and on the other hand the continuity equation $\eqref{NSK Elliptic 1}_1$ yields
    \begin{align*}
        -\TTint \partial_x(\rho u) \varphi^\prime(\rho) \partial_t \rho  \, \dd x
        =
        \TTint |\partial_x(\rho u)|^2 \varphi^\prime(\rho) \, \dd x,
    \end{align*}
    so that together we have
    \begin{align}
        \frac{\mathrm{d}}{\mathrm{d}t} \TTint \rho u \partial_x\varphi(\rho)\, \dd x
        &=
        \TTint \mu \partial_x\varphi(\rho) \partial_{xx} u
        +  \coup \rho \partial_x c \partial_x \varphi(\rho)
        +  |\partial_x(\rho u)|^2 \varphi^\prime (\rho)\nonumber\\
        &\qquad-  \partial_x(\rho u^2) \partial_x\varphi(\rho)
        -  \partial_x \Parti(\rho) \partial_x \varphi(\rho)\, \dd x.\label{prop:BD:proof:eq2}
    \end{align}
    The regularity of the strong solution $(\rho,u,c)$ allows to regularize the continuity equation $\eqref{NSK Elliptic 1}_1$ so that we may infer
    \begin{align}
        \frac{\mathrm{d}}{\mathrm{d}t} \TTint \frac{1}{2} \rho |\partial_x\varphi(\rho)|^2 \, \dd x
        =
        - \TTint \mu \partial_{xx} u \partial_x\varphi(\rho) \, \dd x.\label{prop:BD:proof:eq3}
    \end{align}
    For the technical details and the regularization argument we refer to \cite{Mellet}, where this relation was proven in a more complicated situation, where the viscosity depends on the density.
    Combining $\eqref{prop:BD:proof:eq2}$ and $\eqref{prop:BD:proof:eq3}$, we obtain
    \begin{align}
        \frac{\mathrm{d}}{\mathrm{d}t}
        &\TTint \rho u \partial_x \varphi(\rho) + \frac{1}{2} \rho |\partial_x \varphi(\rho)|^2 \, \dd x+ \TTint \partial_x\Parti(\rho) \partial_x\varphi(\rho) \, \dd x\nonumber\\
        &= \TTint \coup \rho \partial_x c \partial_x\varphi(\rho) + |\partial_x(\rho u)|^2 \varphi^\prime(\rho) - \partial_x(\rho u^2)\partial_x\varphi(\rho) \, \dd x\nonumber \\
        &= \TTint \coup \rho \partial_xc \partial_x\varphi(\rho) + \varphi^\prime(\rho)\bigl( |\partial_x(\rho u)|^2 - \partial_x(\rho u^2)\partial_x \rho \bigr) \, \dd x\nonumber \\
        &= \TTint \coup \rho \partial_x c \partial_x \varphi(\rho) + \varphi^\prime(\rho) \rho^2 |\partial_xu|^2 \, \dd x \nonumber\\
        &= \TTint \coup \rho \partial_x c \partial_x\varphi(\rho) + \mu|\partial_x u|^2 \, \dd x,\label{prop:BD:proof:eq4}
    \end{align}
    where we have used the relation
    \begin{align*}
        |\partial_x(\rho u)|^2 - \partial_x(\rho u^2)\partial_x \rho = \rho^2 |\partial_xu|^2.
    \end{align*}
    Combining $\eqref{prop:BD:proof:eq1}$ and $\eqref{prop:BD:proof:eq4}$ yields
    \begin{align}
        \frac{\mathrm{d}}{\mathrm{d}t} \eta(t) + \TTint \partial_x \Parti(\rho) \partial_x \varphi(\rho) \, \dd x
        = 
        \TTint \coup \rho \partial_x c \partial_x \varphi(\rho) \, \dd x,\label{prop:BD:proof:eq5}
    \end{align}
    where
    \begin{align*}
        \eta(t) := \TTint \frac{1}{2} \rho |u + \partial_x \varphi(\rho)|^2 + W(\rho) + \frac{\coup}{2} |\rho - c|^2 + \frac{\kappa}{2} |\partial_x c|^2 \, \dd x.
    \end{align*}
    Since $(P,\gamma)$ is admissible, we have that $\Parti$ is monotone and therefore
    \begin{align*}
        \partial_x\Parti(\rho) \partial_x\varphi(\rho) = \Parti^\prime(\rho) \varphi^\prime(\rho) |\partial_x \rho|^2 \geq 0,
    \end{align*}
    so that the second term on the left-hand side in $\eqref{prop:BD:proof:eq5}$ is non-negative.
    With Proposition~\ref{Prop:Energy Dissipation}, Proposition~\ref{prop:square control on density} and the conservation of mass we estimate
    \begin{align*}
        \TTint \coup \rho \partial_x c \partial_x\varphi(\rho) \, \dd x
        &\leq
        \coup ||\partial_xc(t)||_{L^\infty(\TT)}\biggl( \TTint \rho + \TTint \rho|\partial_x\varphi(\rho)|^2 \, \dd x\biggr)\\
        &\leq
        C(T_0)\biggl(1 + \TTint \rho|\partial_x\varphi(\rho)|^2 \, \dd x\biggr)\\
        &\leq 
        C(T_0) \biggl( 1 + \TTint \rho |u + \partial_x\varphi(\rho)|^2\, \dd x \biggr)\\
        &\leq
        C(T_0) \biggl( 1 + \eta(t) \biggr),
    \end{align*}
    so that 
    \begin{align*}
        \frac{\mathrm{d}}{\mathrm{d}t} \eta(t) + \TTint \partial_x \Parti(\rho) \partial_x \varphi(\rho) \, \dd x
        \leq
        C(T_0) \biggl( 1 + \eta(t) \biggr).
    \end{align*}
    An application of Gronwall's inequality yields then the relation $\eqref{prop:BD:eq1}$.
    Relation $\eqref{prop:BD:eq1}$ and Proposition~\ref{Prop:Energy Dissipation} imply then
    \begin{align*}
        \TTint \rho |\partial_x \varphi(\rho)|^2\, \dd x \leq C(T_0).
    \end{align*}
    Since 
    \begin{align*}
        \rho |\partial_x \varphi(\rho)|^2 = 4\mu^2 \biggl|\partial_x\biggl(\frac{1}{\sqrt{\rho}}\biggr) \biggr|^2,
    \end{align*}
    we conclude $\eqref{prop:BD:eq2}$. 
\end{proof}

From the BD-entropy inequality, we deduce $L^\infty$-bounds for the density.

\begin{lemma}\label{lem:Linfty density}
    Assume that the hypotheses of Theorem~\ref{Thm:Global Existence and Uniqueness} hold true and let $(\rho,u,c)$ denote the local-in-time strong solution of $\eqref{NSK Elliptic 1}, \eqref{NSK Elliptic initial condition}$ that exists on $[0,T_0)$.\\
    Then we have 
    \begin{align*}
        C(T_0)^{-1} \leq \rho(t,x) \leq C(T_0) \quad \forall \, (t,x) \in [0,T_0) \times \TT.
    \end{align*}
\end{lemma}
\begin{proof}
    We first prove the lower bound.
    To do this, we use the Sobolev embedding $H^1(\TT) \hookrightarrow L^\infty(\TT)$. 
    According to Proposition~\ref{prop:BD}, we only have to show that for any $t \in [0,T_0)$
    \begin{align*}
        \biggl|\biggl|\frac{1}{\sqrt{\rho}}(t)\biggr|\biggr|_{L^2(\TT)} \leq C(T_0).
    \end{align*}
    Using the Poincaré inequality and Proposition~\ref{prop:BD}, it suffices to prove
    \begin{align*}
        \EE\biggl[\frac{1}{\sqrt{\rho}}(t)\biggr] \leq C(T_0).
    \end{align*}
    In the following calculations we omit the argument for $t$. 
    Conservation of mass leads to
    \begin{align*}
        M_0^{-1} \leq \TTint \rho = \TTint \rho_0 \leq M_0.
    \end{align*}
    This implies with Hölder's inequality
    \begin{align}\label{lem:Linfty density:proof:eq1}
        \TTint \rho \frac{1}{\sqrt{\rho}} \, \dd x = \TTint \sqrt{\rho} \, \dd x \leq \sqrt{M_0} \leq C(T_0).
    \end{align}
    Together with Proposition~\ref{prop:square control on density}, Poincaré's inequality and Proposition~\ref{prop:BD}, this implies
    \begin{align*}
        \TTint \rho \biggl( \frac{1}{\sqrt{\rho}} - \EE\biggl[\frac{1}{\sqrt{\rho}}\biggr]  \biggr)\, \dd x
        \leq 
        ||\rho(t)||_{L^2(\TT)} \biggl|\biggl|\frac{1}{\sqrt{\rho}} - \EE\biggl[\frac{1}{\sqrt{\rho}}\biggr]\biggr|\biggr|_{L^2(\TT)}
        \leq C(T_0).
    \end{align*}
    Finally, we obtain from the above relations that
    \begin{align*}
        M_0^{-1} \EE\biggl[\frac{1}{\sqrt{\rho}}\biggr]
        \leq
        \TTint \rho \EE\biggl[\frac{1}{\sqrt{\rho}}\biggr] \, \dd x
        =
        \TTint \rho \frac{1}{\sqrt{\rho}} \, \dd x
        -\TTint \rho\biggl( \frac{1}{\sqrt{\rho}} - \EE\biggl[\frac{1}{\sqrt{\rho}}\biggr] \biggr) \, \dd x
        \leq C(T_0),
    \end{align*}
    which implies
    \begin{align*}
        \EE\biggl[\frac{1}{\sqrt{\rho}}\biggr] \leq C(T_0).
    \end{align*}
    The proof for the lower bound is complete.
    For the upper bound, we use the Sobolev embedding $W^{1,1}(\TT) \hookrightarrow L^\infty(\TT)$, so that we only have to prove
    \begin{align*}
        ||\sqrt{\rho}||_{L^1(\TT)} + ||\partial_x \sqrt{\rho}||_{L^1(\TT)} \leq C(T_0).
    \end{align*}
    We have already estimated the first term in $\eqref{lem:Linfty density:proof:eq1}$. For the second term, we use
    Proposition~\ref{prop:square control on density} and Proposition~\ref{prop:BD} to estimate
    \begin{align*}
        2||\partial_x\sqrt{\rho}||_{L^1(\TT)}
        \TTint \rho \biggl|\partial_x\biggl(\frac{1}{\sqrt{\rho}}\biggr)\biggr| \, \dd x
        \leq ||\rho(t)||_{L^2(\TT)} \biggl|\biggl|\partial_x\biggl( \frac{1}{\sqrt{\rho}} \biggr) \biggr|\biggr|_{L^2(\TT)} 
        \leq C(T_0).
    \end{align*}
\end{proof}

With the $L^\infty$-bounds for the density, we come back to the BD-entropy inequality and obtain an uniform-in-time $H^1$-bound for $\rho$.

\begin{lemma}\label{uniform-in-time-H^1}
    Assume that the hypotheses of Theorem~\ref{Thm:Global Existence and Uniqueness} hold true and let $(\rho,u,c)$ denote the local-in-time strong solution of $\eqref{NSK Elliptic 1}, \eqref{NSK Elliptic initial condition}$ that exists on $[0,T_0)$.\\
    Then we have
    \begin{align*}
        ||\rho||_{L^\infty(0,T_0;H^1(\TT))} \leq C(T_0).
    \end{align*}
\end{lemma}
\begin{proof}
    With Proposition~\ref{prop:square control on density} we only have to show that
    \begin{align*}
        ||\partial_x \rho||_{L^\infty(0,T_0;L^2(\TT))} \leq C(T_0).
    \end{align*}
    This can be done by using Proposition~\ref{prop:BD} and Lemma~\ref{lem:Linfty density}
    \begin{align*}
        \TTint |\partial_x\rho|^2
        =
        \TTint \rho^3 \biggl| \frac{\partial_x \rho}{\rho^{\frac{3}{2}}}\biggr|^2\,\dd x
        \leq C(T_0) \biggl|\biggl|\partial_x\biggl(\frac{1}{\sqrt{\rho}}\biggr)\biggr|\biggr|_{L^2(\TT)}^2 
        \leq C(T_0).
    \end{align*}
\end{proof}

Finally, we prove an uniform-in-time $H^1$-bound for the velocity $u$.

\begin{proposition}\label{prop:uniform-in-time velocity}
    Assume that the hypotheses of Theorem~\ref{Thm:Global Existence and Uniqueness} hold true and let $(\rho,u,c)$ denote the local-in-time strong solution $\eqref{NSK Elliptic 1},\eqref{NSK Elliptic initial condition}$ that exists on $[0,T_0)$.\\
    Then we have
    \begin{align*}
        ||u||_{L^\infty(0,T_0;H^1(\TT))} \leq C(T_0).
    \end{align*}
\end{proposition}
\begin{proof}
    We have 
    \begin{align*}
        \rho(\partial_t u + u \partial_x u) = \mu\partial_{xx}u - \partial_x\Parti(\rho) + \coup \rho \partial_x c
        \quad \text{a.e. in } (0,T_0) \times \TT.
    \end{align*}
    Both sides are in $L^2(0,T_0;L^2(\TT))$.
    Hence, we may multiply both sides by $\partial_{xx}u \in L^2(0,T_0;L^2(\TT))$ and integrate in space and time to obtain
    \begin{align}\label{prop:uniform-in-time velocity:proof:eq1}
        \int_0^t \TTint \rho(\partial_t u + u \partial_x u)\partial_{xx}u \, \dd x \, \dd \tau
        =
        \int_0^t \TTint \mu |\partial_{xx}u|^2 - \partial_x \Parti(\rho) \partial_{xx}u + \coup \rho \partial_x c \partial_{xx}u \, \dd x \, \dd \tau.
    \end{align}
    After an approximation argument using integration by parts, the left hand side can be rewritten as
    \begin{align*}
        &\int_0^t \TTint \rho(\partial_t u + u \partial_x u)\partial_{xx}u\, \dd x \, \dd \tau\\
        &=
        -\frac{1}{2}\TTint \rho(t,x) |\partial_xu(t,x)|^2 \, \dd x
        + \frac{1}{2} \TTint \rho_0 |\partial_xu_0|^2 \, \dd x\\
        &\qquad- 
        \int_0^t \TTint \partial_x \rho (\partial_t u + u \partial_x u) \partial_x u + \rho |\partial_x u|^2\partial_x u  \, \dd x \, \dd \tau.
    \end{align*}
    Plugging this relation into \eqref{prop:uniform-in-time velocity:proof:eq1} yields 
    \begin{align*}
        &\frac{1}{2} \TTint \rho(t,x) |\partial_x u(t,x)|^2 \, \dd x+ \int_0^t \TTint \mu |\partial_{xx}u|^2 \, \dd x \, \dd \tau\\
        &=
        \frac{1}{2} \TTint \rho_0 |\partial_x u_0|^2  \, \dd x
        +
        \int_0^t \TTint \partial_x \Parti(\rho) \partial_{xx}u 
        - \int_0^t \TTint \partial_x \rho\bigl( \partial_t u + u \partial_x u \bigr)\partial_x u \, \dd x \, \dd \tau\\
        &\qquad- \int_0^t\TTint  \rho |\partial_x u|^2 \partial_x u \, \dd x \, \dd \tau  
        -\int_0^t \TTint \coup \rho \partial_x c \partial_{xx}u \, \dd x \, \dd \tau
    \end{align*}
    We estimate the right hand side with the help of Proposition~\ref{uniform-in-time-H^1}. 
    For $\varepsilon>0$, let us denote by $C(\varepsilon)>0$ a generic positive constant that may vary from line to line but only depends on $\varepsilon$.
    First we obtain with Young's inequality and Proposition~\ref{uniform-in-time-H^1}
    \begin{align*}
        \int_0^t \TTint \partial_x \Parti(\rho) \partial_{xx}u \, \dd x \, \dd \tau
        &\leq
        \frac{1}{4\mu\varepsilon} \int_0^t ||\partial_x \Parti(\rho)||_{H^1(\TT)}^2  \, \dd \tau
        +
        \varepsilon\int_0^t \TTint \mu |\partial_{xx} u|^2 \, \dd x \, \dd \tau \\
        &\leq 
        C(T_0)C(\varepsilon)\int_0^t 1 \, \dd \tau + \varepsilon \int_0^t \TTint \mu |\partial_{xx}u|^2 \, \dd x \, \dd \tau.
    \end{align*}
    Next, we use the relation
    \begin{align*}
        \partial_x \rho (\partial_t u + u \partial_x u) 
        =
        \frac{\partial_x \rho}{\rho} (\mu \partial_{xx}u - \partial_x\Parti(\rho) + \coup \rho \partial_x c),
    \end{align*}
    to obtain with Young's inequality
    \begin{align*}
        \biggl| \TTint \partial_x \rho (\partial_t u + u \partial_x u ) \partial_x u \, \dd x\biggr|
        &\leq
        \sqrt{\mu}\biggl|\biggl|\frac{1}{\rho}\biggr|\biggr|_{L^\infty(\TT)} ||\rho||_{H^1(\TT)} ||\partial_x u||_{L^\infty(\TT)} \biggl( \TTint \mu |\partial_{xx}u|^2 \, \dd x\biggr)^{\frac{1}{2}}\\
        &\qquad+
        \biggl|\biggl|\frac{1}{\rho}\biggr|\biggr|_{L^\infty(\TT)} ||\rho||_{H^1(\TT)} ||\partial_x \Parti(\rho)||_{L^2(\TT)} ||\partial_x u||_{L^\infty(\TT)}\\
        &\qquad+
        \coup ||\partial_x \rho ||_{H^1(\TT)} ||\partial_x c||_{L^\infty(\TT)} ||\partial_xu||_{L^2(\TT)}\\
        &\leq
        C(T_0) C(\varepsilon)\biggl(1 + ||\partial_x u||^2_{L^2(\TT)} \biggr) 
        +
        \varepsilon \TTint \mu |\partial_{xx}u|^2 \, \dd x,
    \end{align*}
    where we have used Proposition~\ref{uniform-in-time-H^1} to bound the $H^1$-norms of $\rho$, the control on $\partial_x c$ that is provided by Proposition~\ref{prop:square control on density} and the interpolation inequality
    \begin{align*}
        ||\partial_x u||_{L^\infty(\TT)} \leq C ||\partial_x u||_{L^2(\TT)}^{\frac{1}{2}} ||\partial_{xx}u||_{L^2(\TT)}^\frac{1}{2},
    \end{align*}
    for some constant $C>0$.
    This relation yields
    \begin{align*}
        &\biggl| \int_0^t \TTint \partial_x \rho ( \partial_t u + u \partial_x u) \partial_x u \, \dd x \, \dd \tau\biggr|\\
        &\leq
        C(T_0)C(\varepsilon)\int_0^t \biggl(1 + ||\partial_x u||_{L^2(\TT)}^2 \biggr) \, \dd \tau 
        +
        \varepsilon\int_0^t \TTint \mu |\partial_{xx}u|^2 \, \dd x \, \dd \tau .
    \end{align*}
    By similar arguments, we find
    \begin{align*}
        \biggl| \int_0^t \TTint \rho |\partial_x u|^2 \partial_x u \, \dd x \, \dd \tau\biggr|
        \leq
        C(T_0)C(\varepsilon) \int_0^t \biggl(1 + ||\partial_x u||^4_{L^2(\TT)} \biggr) \, \dd \tau
        +
        \varepsilon\int_0^t \TTint \mu|\partial_{xx}u|^2 \, \dd x \, \dd \tau.
    \end{align*}
    Finally, we estimate with Proposition~\ref{prop:square control on density}
    \begin{align*}
        \biggl| \int_0^t \TTint \coup \rho \partial_x c \partial_{xx }u \, \dd x \, \dd \tau\biggr|
        &\leq
        C(T_0)C(\varepsilon) \int_0^t \TTint |\partial_x c|^2 \, \dd x \, \dd \tau
        +
        \varepsilon\int_0^t \TTint \mu |\partial_{xx}u|^2 \, \dd x \, \dd \tau \\
        &\leq C(T_0) C(\varepsilon)\int_0^t 1\, \dd \tau+ \varepsilon\int_0^t \TTint \mu |\partial_{xx}u|^2 \, \dd x \, \dd \tau.
    \end{align*}
    Choosing $\varepsilon= \frac{1}{8}$, we obtain from the preceding estimates the relation
    \begin{align*}
        &\frac{1}{2} \TTint \rho(t,x) |\partial_x u(t,x)|^2 \, \dd x
        +
        \frac{1}{2} \int_0^t \TTint \mu|\partial_{xx}u|^2 \, \dd x \, \dd \tau\\
        &\leq
        \frac{1}{2} \TTint \rho_0|\partial_x u_0|^2  \, \dd x
        +
        C(T_0) \int_0^t\biggl( 1 +  ||\partial_x u||^4_{L^2(\TT)}\biggr) \, \dd \tau\\
        &\leq
        C(T_0) + C(T_0) \int_0^t\biggl(1 +  ||\partial_x u||^4_{L^2(\TT)} \biggr) \, \dd \tau.
    \end{align*}
    Using Proposition~\ref{Prop:Energy Dissipation} and applying Gronwall's inequality yields
    \begin{align*}
        ||\partial_x u||_{L^\infty(0,T_0;L^2(\TT))} \leq C(T_0).
    \end{align*}
    Together with Proposition~\ref{Prop:Energy Dissipation}, this yields the claim.
\end{proof}


\section{Refined A-Priori Estimates}\label{Refined A-Priori Estimates}

In this section, we prove an a-priori estimate concerning the effective viscous flux that is necessary to perform the homogenization procedure in Section~\ref{Homogenization}.
Since we will assume strong oscillations in the initial densities during the homogenization process, it is crucial that this a-priori estimate does only depend on $M_0, ||u_0||_{H^1(\TT)}, \mu, \coup ,\kappa$, but not on quantities that involve the derivative of the density.
Let us recall that we have already proven two such a-priori estimates in Section~\ref{Construction of Global Strong Solutions}, namely Proposition~\ref{Prop:Energy Dissipation} and Proposition~\ref{prop:square control on density}.
However, these a-priori estimates are not strong enough to control the homogenization procedure.
Following \cite{Hillairet_New_Physical}, we introduce for a global-in-time strong solution $(\rho,u,c)$ of $\eqref{NSK Elliptic 1}, \eqref{NSK Elliptic initial condition}$ the effective viscous flux as
\begin{align*}
    \Sigma := \mu \partial_x u - \Parti(\rho)
\end{align*}
and accordingly the initial effective viscous flux as
\begin{align*}
    \Sigma_0 := \mu \partial_x u_0 - \Parti(\rho_0).
\end{align*}
We then prove the following estimate.

\begin{theorem}\label{Thm:Refined A Priori Estimates}
    Assume that the hypotheses of Theorem~\ref{Thm:Global Existence and Uniqueness} hold true and let $(\rho,u,c)$ denote the global-in-time strong solution of $\eqref{NSK Elliptic 1}, \eqref{NSK Elliptic initial condition}$.\\
    Then there exist some time $T_0>0$ and some constant $C_0>0$, both only depending on $M_0, ||u_0||_{H^1(\TT)},\coup,\kappa, \mu$, such that we have
    \begin{align}\label{Thm:Refined A Priori Estimate:Result1}
        (2M_0)^{-1} \leq \rho(t,x) \leq 2M_0 \quad \forall \, (t,x) \in [0,T_0]\times\TT,
    \end{align}
    and 
    \begin{align}\label{Thm:Refined A Priori Estimate:Result2}
        ||u||_{L^\infty(0,T_0;H^1(\TT))}^2
        +
        ||\partial_x \Sigma||_{L^2(0,T_0;L^2(\TT))}^2
        \leq
        C_0.
    \end{align}
\end{theorem}

The proof of Theorem~\ref{Thm:Refined A Priori Estimates} follows the lines in \cite{Hillairet_New_Physical} via two lemmata. To shorten the notation, let us introduce the following quantities:

\begin{align*}
    \Kparti &:= \max\limits_{\lambda \in [(2M_0)^{-1}, 2M_0]} \bigl\{ |\Parti(\lambda)| \bigr\},\\
    \Ku &:= \frac{8}{\mu^2}\biggl(
        2 + 5 M_0 \mu
        \biggr)
        \biggl(||\Sigma_0||_{L^2(\TT)}^2 +1 + |\Kparti|^2\biggr),\\
    \Kd &:= \frac{1}{\mu} \biggl[ \Csob \sqrt{2 \mu E_0 + \Ku + 2 |\Kparti|^2 } + \Kparti \biggr].
\end{align*}

Here, $\Csob>0$ denotes the constant of the Sobolev embedding $H^1(\TT) \hookrightarrow L^\infty(\TT)$.
We emphasize, that $\Kparti, \Ku$ and $\Kd$ only depend on $M_0,||u_0||_{H^1(\TT)},\mu,\coup$ and $\kappa$. 
Being in the strong solution framework, we may apply a continuity argument to find some small time $\tilde{T}_0 \in (0,\infty)$, such that we have
\begin{align}\label{A Priori Estimates: Local-in-Time Bound density}
    (2M_0)^{-1} \leq \rho(t,x) \leq 2M_0 \quad \forall \, (t,x) \in [0,\tilde{T}_0] \times \TT,
\end{align}
and
\begin{align}\label{A Priori Estimate: Local-in-Time Bound effective pressure}
    ||\partial_x u||_{L^\infty(0,\tildeT;L^2(\TT))}^2
    +
    ||\partial_x \Sigma||_{L^2(0,\tildeT,L^2(\TT)}^2
    \leq \Ku.
\end{align}
In order to conclude the proof of Theorem~\ref{Thm:Refined A Priori Estimates}, we have to show that $\tildeT$ in fact only depends on the allowed quantities.
This will be done as follows:
We show that there exists some time $T_0 \in (0,\infty)$, only depending on the allowed quantities, such that if $T_0 < \tildeT$, we have in fact sharper versions of the inequalities $\eqref{A Priori Estimates: Local-in-Time Bound density}$ and $\eqref{A Priori Estimate: Local-in-Time Bound effective pressure}$ on $[0,T_0]$.
Then, via a connectedness argument we prove that this already implies $\tildeT=T_0$.
To prove the sharper estimates, we proceed as in \cite{Hillairet_New_Physical} via two lemmata.
For the first lemma, we can use the arguments demonstrated in \cite{Hillairet_New_Physical}, since these only rely on the continuity equation $\eqref{NSK Elliptic 1}_1$ and an a-priori estimate for the derivative of the velocity, that is provided by Proposition~\ref{Prop:Energy Dissipation}.
Therefore, we will refer for details concerning the proof to \cite{Hillairet_New_Physical}.
However, the second lemma relies on the momentum equation.
Therefore, we have to modify the proof given in \cite{Hillairet_New_Physical}. The crucial ingredient making our modification work is the control on the order parameter provided by Proposition~\ref{prop:square control on density}.

\begin{lemma}\label{Lemma1: Refined Velocity and Density Bound}
    Assume that the hypotheses of Theorem~\ref{Thm:Global Existence and Uniqueness} hold true and let $(\rho,u,c)$ denote the global-in-time strong solution of $\eqref{NSK Elliptic 1}, \eqref{NSK Elliptic initial condition}$.
    Suppose that for some $\tildeT \in (0,1)$, the inequalities $\eqref{A Priori Estimates: Local-in-Time Bound density}$ and $\eqref{A Priori Estimate: Local-in-Time Bound effective pressure}$ hold on $[0,\tilde{T_0}]$.\\
    Then there exists some $T_\rho \in (0,\infty)$, only depending on $M_0, ||u_0||_{H^1(\TT)},\mu,\coup,\kappa$, such that if $\tildeT<T_\rho$, we have
    \begin{align*}
        ||\partial_x u||_{L^1(0,\tildeT;L^\infty(\TT))} \leq \sqrt{\tildeT} \Kd,
    \end{align*}
    and
    \begin{align}\label{Lemma 1: Refined Velocity and Density Bound: Result 1}
        \frac{2}{3M_0} \leq \rho(t,x) \leq \frac{3}{2}M_0 \quad \forall \, (t,x) \in [0,\tildeT] \times \TT.
    \end{align}
\end{lemma}
\begin{proof}
    Both estimates follow from Proposition~\ref{Prop:Energy Dissipation} by the arguments presented in the proof of Proposition 6 and Proposition 7 in \cite{Hillairet_New_Physical}.
\end{proof}


For the second lemma, we also follow the proof in \cite{Hillairet_New_Physical}. 
This proof relies on the momentum equation of the compressible Navier--Stokes equations. 
Using the artificial pressure function, we enter the framework of \cite{Hillairet_New_Physical}, but we have to deal with the additional term $\coup \rho \partial_x c$ in the momentum equation.
However, using the control provided by Proposition~\ref{prop:square control on density}, the method of proof given in \cite{Hillairet_New_Physical} also applies for our situation.

\begin{lemma}\label{Lemma 3: Better Bound on the Effective Pressure}
    Assume that the hypotheses of Theorem~\ref{Thm:Global Existence and Uniqueness} hold true and let $(\rho,u,c)$ denote the global-in-time strong solution of $\eqref{NSK Elliptic 1}, \eqref{NSK Elliptic initial condition}$.
    Suppose that for some $\tildeT \in (0,1)$, the inequalities $\eqref{A Priori Estimates: Local-in-Time Bound density}$ and $\eqref{A Priori Estimate: Local-in-Time Bound effective pressure}$ hold on $[0,\tilde{T_0}]$.\\
    Then there exists some time $T_u \in (0,\infty)$ only depending on $M_0, ||u_0||_{H^1(\TT)}, \mu,\coup,\kappa$, such that if $\tildeT<T_u$, we have
    \begin{align}\label{Lemma 3: Better Bound on the Effective Pressure:Result1}
        ||\partial_x u||_{L^\infty(0,\tildeT;L^2(\TT))}^2
        +
        ||\partial_x \Sigma||_{L^2(0,\tildeT;L^2(\TT))}^2
        \leq
        \frac{4}{\mu^2}
        \biggl(
        2 + 5M_0\mu
        \biggr)
        \biggl(||\Sigma_0||_{L^2(\TT)}^2 + 1 + |\Kparti|^2\biggr).
    \end{align}
\end{lemma}
\begin{proof}
    Throughout this proof, we denote by $\mathcal{C}>0$ a generic positive constant that may vary from line to line but only depends on $M_0, ||u_0||_{H^1(\TT)}, \mu ,\coup ,\kappa $.
    Being in the strong solution framework, it is easy to verify that $\rho$ satisfies the continuity equation $\eqref{NSK Elliptic 1}_1$ in the renormalized sense, that is, we have for any $b \in C^1([0,\infty))$, the equation
    \begin{align}\label{Lemma 3: Better Bounds on the effective Pressure: Proof 1}
        \partial_t b(\rho) + \partial_x(b(\rho) u ) - (b^\prime(\rho) \rho - b(\rho)) \partial_x u = 0 \quad \text{a.e. on } (0,\tildeT)\times \TT.
    \end{align}
    Using this equation with $b = \Parti$, we deduce that $\Parti(\rho)$ satisfies
    \begin{align*}
        \Parti(\rho) \in H^1(0,\tildeT;L^2(\TT)) \cap C([0,\tildeT];C(\TT)) \cap L^2(0,\tildeT;H^1(\TT)).
    \end{align*}
    We fix $T \in (0,\tildeT)$ and define on $(0,T)$ a multiplier via
    \begin{align*}
        m := \mu \partial_t u - \prim\bigl[ \partial_t \Parti(\rho) - \EE[\partial_t \Parti(\rho)] \bigr] \in L^2(0,T;L^2(\TT)).
    \end{align*}
    In the sequel, we omit the argument of $\Parti(\rho)$ for the sake of brevity. 
    We rewrite the momentum equation as
    \begin{align*}
        \rho( \partial_t u + u \partial_x u - \coup \partial_x c )
        = \partial_x \Sigma.
    \end{align*}
    Then, we multiply this equation by $m$ and integrate over space and time to obtain
    \begin{align}\label{Lemma 3: Better Bounds on the effective Pressure: Proof 2}
        \int_0^T \TTint \rho(\partial_t u + u \partial_x u - \coup \partial_x c) m \, \dd x \, \dd t
        =
        \int_0^T \TTint \partial_x \Sigma m \, \dd x  \, \dd t.
    \end{align}
    We notice, that formally $\partial_x m = \partial_t \Sigma + \EE[\partial_t \Parti]$.
    We are not allowed to integrate by parts on the right hand side due to the lack of spatial regularity for $\partial_t u$. 
    However, we may perform a Fourier series approximation of $u$.
    By the regularity of $u$, this approximation is strong enough to show that the right hand side of equation $\eqref{Lemma 3: Better Bounds on the effective Pressure: Proof 2}$ can be rewritten as
    \begin{equation}\label{Lemma 3: Better Bounds on the effective Pressure: Proof 3}
    \begin{aligned}
            \int_0^T \TTint \partial_x\Sigma  m \, \dd x \, \dd t
            =
            -
            \frac{1}{2}\TTint |\Sigma(T,\cdot)|^2 \,  \dd x
            +
            \frac{1}{2}\TTint |\Sigma_0|^2 \,  \dd x
            -
            \int_0^T \EE(\Sigma)\TTint \partial_t \Parti \, \dd x\,  \dd t,
    \end{aligned}
    \end{equation}
    For details concerning the approximation argument we refer to \cite{Hillairet_New_Physical}.
    Concerning the left hand side of equation $\eqref{Lemma 3: Better Bounds on the effective Pressure: Proof 2}$ a straightforward application of Young's inequality yields
    \begin{align*}
        &\int_0^T \TTint \rho ( \partial_t u + u \partial_x u - \coup \partial_x c ) 
        \prim \bigl[ \partial_t \Parti - \EE[\partial_t \Parti] \bigr]\\
        &\leq
        \frac{\mu}{2} \int_0^T \TTint \rho ( |\partial_t u|^2 + |u\partial_x u|^2 + |\coup \partial_x c|^2 )
        +
        \frac{3}{2\mu} \int_0^T \TTint \rho \Bigl| \prim \bigl[ \partial_t \Parti - \EE[ \partial_t \Parti ] \bigr] \Bigr|^2,
    \end{align*}
    which, by virtue of the elementary calculation
    \begin{align*}
        \mu \rho (\partial_t u + u \partial_x u - \coup \partial_x c) \partial_t u
        =&
        \frac{\mu}{2} \rho | \partial_t u + u \partial_x u - \coup \partial_x c |^2 
        +
        \frac{\mu}{2} \rho ( |\partial_t u|^2 - |u\partial_x u|^2 - |\coup \partial_x c|^2 )\\
        &\qquad+ 
        \mu\coup \rho u \partial_x u \partial_x c,
    \end{align*}
    results in
    \begin{align}
            &\int_0^T \TTint \rho ( \partial_t u + u \partial_x u - \coup \partial_x c) m\ \, \dd x \, \dd t \nonumber\\
            &\geq
            \frac{\mu}{2} \int_0^T \TTint \rho |\partial_t u + u \partial_x u - \coup\partial_x c|^2 \, \dd x \, \dd t
            - 
            \mu \int_0^T \TTint \rho |u \partial_x u|^2 \, \dd x \, \dd t\nonumber\\
            &\qquad-
            \mu \coup^2 \int_0^T \TTint \rho |\partial_x c|^2 \, \dd x \, \dd t
            +
            \mu\coup \int_0^T \TTint \rho u \partial_x u \partial_x c \, \dd x \, \dd t\nonumber\\
            &\qquad-
            \frac{3}{2 \mu} \int_0^T \TTint \rho \Bigl|\prim\bigl[ \partial_t \Parti - \EE[\partial_t \Parti]\bigl] \Bigr|^2\, \dd x \, \dd t\nonumber\\
            &\geq
            \frac{\mu}{4 M_0}  \int_0^T \TTint |\partial_x \Sigma|^2 \, \dd x \, \dd t
            - 
            \mu \int_0^T \TTint \rho |u \partial_x u|^2  \, \dd x \, \dd t
            -
            \mu \coup^2 \int_0^T \TTint \rho |\partial_x c|^2\, \dd x \, \dd t\nonumber\\
            &\qquad+
            \mu \coup \int_0^T \TTint \rho u \partial_x u \partial_x c \, \dd x \, \dd t
            -
            \frac{3}{2\mu} \int_0^T \TTint \rho \Bigl| \prim\bigl[ \partial_t \Parti - \EE[ \partial_t \Parti ] \bigr] \Bigr|^2\, \dd x\, \dd t.\label{Lemma 3: Better Bounds on the effective Pressure: Proof 4}
    \end{align}
    In the last inequality, we have used the bound $\eqref{A Priori Estimates: Local-in-Time Bound density}$ for the density $\rho$ and the momentum equation $\eqref{NSK Elliptic 1}_2$. 
    Combining the equations $\eqref{Lemma 3: Better Bounds on the effective Pressure: Proof 2}$, $\eqref{Lemma 3: Better Bounds on the effective Pressure: Proof 3}$ and $\eqref{Lemma 3: Better Bounds on the effective Pressure: Proof 4}$ yields
    \begin{align}
            &\frac{1}{2}||\Sigma(T,\cdot)||_{L^2(\TT)}^2
            +
            \frac{\mu}{4M_0}  ||\partial_x\Sigma||_{L^2(0,T;L^2(\TT))}^2\nonumber\\
            &\leq
            \frac{1}{2} ||\Sigma_0||_{L^2(\TT)}^2
            -
            \int_0^T \EE(\Sigma) \biggl( \TTint \partial_t \Parti \, \dd x \biggr) \,\dd t
            +
            \mu \int_0^T \TTint \rho |u\partial_x u|^2 \,\dd x\, \dd t\nonumber\\
            &\qquad+
            \mu \coup^2 \int_0^T \TTint \rho |\partial_x c|^2 \,\dd x \,\dd t
            -
            \mu \coup \int_0^T \TTint \rho u \partial_x u \partial_x c \, \dd x \,\dd t \nonumber\\
            &\qquad+
            \frac{3}{2\mu} \int_0^T \TTint \rho \Bigl| \prim\bigl[ \partial_t \Parti - \EE[\partial_t \Parti] \bigr] \Bigr|^2\, \dd x \, \dd t\nonumber\\
            &=: \frac{1}{2} ||\Sigma_0||_{L^2(\TT)}^2 + \sum\limits_{i=1}^5 I_i.\label{Lemma 3: Better Bounds on the effective Pressure: Proof 5}
    \end{align}
    We estimate each term on the right hand side of inequality $\eqref{Lemma 3: Better Bounds on the effective Pressure: Proof 5}$ separately.
    For $I_1$, we use the renormalized continuity equation $\eqref{Lemma 3: Better Bounds on the effective Pressure: Proof 1}$, to rewrite $I_1$ as
    \begin{align*}
        I_1 = - \int_0^T \EE[\Sigma] \TTint \Bigl(\Parti^\prime(\rho) \rho - \Parti(\rho)\Bigr) \partial_x u \, \dd x \, \dd t.
    \end{align*}
    Young's inequality and Jensen's inequality provide us
    \begin{align*}
        |I_1| 
        \leq
        \mathcal{C} \int_0^T ||\Sigma(t,\cdot)||_{L^2(\TT)}^2\, \dd t + \mathcal{C} \int_0^T ||\partial_x u(t,\cdot)||_{L^2(\TT)}^2 \,\dd t.
    \end{align*}
    For $I_2$, we use $\eqref{A Priori Estimates: Local-in-Time Bound density}$ and the Sobolev embedding $H^1(\TT) \hookrightarrow L^\infty(\TT)$ together with Proposition~\ref{Prop:Energy Dissipation} to estimate
    \begin{align*}
        |I_2|
        &\leq
        \mathcal{C} \int_0^T ||u(t,\cdot)||_{L^\infty(\TT)}^2 ||\partial_x u(t,\cdot)||_{L^2(\TT)}^2 \, \dd t\\
        &\leq 
        \mathcal{C} \int_0^T 
        \bigl(1 + ||\partial_x u(t,\cdot)||_{L^2(\TT)}^2\bigr) 
        ||\partial_x u(t,\cdot)||_{L^2(\TT)}^2 
        \, \dd t
    \end{align*}
    By virtue of $\partial_x u = \frac{\Sigma + \Parti}{\mu}$ and since $\rho$ is bounded via $\eqref{A Priori Estimates: Local-in-Time Bound density}$, we may estimate $|I_2|$ further as
    \begin{align*}
        |I_2| 
        \leq
        \mathcal{C} \int_0^T 
        \bigl( 1 + ||\partial_x u(t,\cdot)||_{L^2(\TT)}^2 \bigr) \, \dd t
        +
        \mathcal{C} \int_0^T
        \bigl( 1 + ||\partial_x u(t,\cdot)||_{L^2(\TT)}^2 \bigr)
        ||\Sigma(t,\cdot)||_{L^2(\TT)}^2 \, \dd t.
    \end{align*}
    Regarding $I_3$ and $I_4$ we use Proposition~\ref{Prop:Energy Dissipation} and Proposition~\ref{prop:square control on density} to estimate
    \begin{align*}
        |I_3| \leq \mathcal{C} \int_0^T 1 \, \dd t,
    \end{align*}
    and
    \begin{align*}
        |I_4| 
        &\leq
        \mathcal{C} \int_0^T \TTint |u \partial_x u|\, \dd x \, \dd t
        \leq
        \mathcal{C} \int_0^T \TTint |u|^2 \, \dd x \, \dd t + \mathcal{C} \int_0^T \TTint |\partial_x u|^2 \, \dd x \, \dd t\\
        &\leq
        \mathcal{C} \int_0^T \bigl( 1 + ||\partial_x u(t,\cdot)||^2_{L^2(\TT)} \bigr) \, \dd t.
    \end{align*}
    Finally, for $I_5$, we use the renormalized continuity equation $\eqref{Lemma 3: Better Bounds on the effective Pressure: Proof 1}$ to obtain
    \begin{align*}
        \partial_t \Parti - \EE[\partial_t \Parti]
        =
        -\partial_x(\Parti u) 
        - 
        \biggl(
        (\Parti^\prime \rho - \Parti) \partial_x u
        -
        \EE[(\Parti^\prime \rho - \Parti)\partial_x u]
        \biggr),
    \end{align*}
    so that its mean free primitive computes to
    \begin{align*}
        \prim \bigl[
        \partial_t \Parti - \EE[\partial_t \Parti] \bigr]
        =
        -\bigl(
        \Parti u - \EE[\Parti u]
        \bigr)
        - \prim \bigl[
        (\Parti^\prime \rho - \Parti) \partial_x u
        -
        \EE[(\Parti^\prime \rho - \Parti)\partial_x u]
        \bigr].
    \end{align*}
    Then, using Poincaré's inequality, Jensen's inequality, Proposition~\ref{Prop:Energy Dissipation} and the bound $\eqref{A Priori Estimates: Local-in-Time Bound density}$, we deduce
    \begin{align*}
        |I_5|
        \leq
        \mathcal{C} \int_0^T \bigl( 1 + ||\partial_x u(t,\cdot)||_{L^2(\TT)}^2 \bigr)\, \dd t.
    \end{align*}
    In total, our estimates for $I_1, I_2, I_3, I_4,I_5$ together with inequality $\eqref{Lemma 3: Better Bounds on the effective Pressure: Proof 5}$ imply that
    \begin{align}\label{Lemma 3: Better Bounds on the effective Pressure: Proof 6}
        &||\Sigma(T,\cdot)||_{L^2(\TT)}^2
        +
        \frac{\mu}{4M_0} ||\partial_x \Sigma||_{L^2(0,T;L^2(\TT))}^2\nonumber\\
        &\leq
        ||\Sigma_0||_{L^2(\TT)}^2
        +
        \int_0^T f(t) \, \dd t 
        +
        \int_0^T f(t) ||\Sigma(t,\cdot)||_{L^2(\TT)}^2 \, \dd t,
    \end{align}
    with
    \begin{align*}
        f(t):= \mathcal{C} \bigl( 1 + ||\partial_x u(t,\cdot)||_{L^2(\TT)}^2 \bigr).
    \end{align*}
    By virtue of the inequality $\eqref{A Priori Estimate: Local-in-Time Bound effective pressure}$, we have
    \begin{align*}
        \int_0^T f(t) \dd t 
        = \int_0^T \mathcal{C} (1 + ||\partial_x u(t,\cdot)||_{L^2(\TT)}^2 )\, \dd t
        \leq
        \mathcal{C}(1 + \Ku) T.
    \end{align*}
    Since $\Ku$ only depends on $M_0,||u_0||_{H^1(\TT)},\mu,\coup$ and $\kappa$, we find some small time $T_u \in (0,\infty)$ only depending on these quantities, such that if $\tildeT \leq  T_u$, we have for any $T \in (0,\tildeT)$
    \begin{align}\label{Lemma 3: Better Bounds on the effective Pressure: Inserted}
        \int_0^{T} f(t) \, \dd t 
        \leq
        \min\biggl\{ ||\Sigma_0||_{L^2(\TT)}^2 + 1, \ln(2) \biggr\}.
    \end{align}
    Let us assume $\tildeT<T_u$.
    Then relation $\eqref{Lemma 3: Better Bounds on the effective Pressure: Inserted}$ in combination with Gronwall's inequality yields for any $T \in (0,\tildeT)$ 
    \begin{align}\label{Lemma 3: Better Bounds on the effective Pressure: Proof 7}
        ||\Sigma(T,\cdot)||_{L^2(\TT)}^2
        &\leq
        ||\Sigma_0||_{L^2(\TT)}^2 \exp\biggl(\int_0^T f(t) \, \dd t \biggr) + \int_0^T f(t) \exp\biggl( \int_t^Tf(s) \, \dd s \biggr)\, \dd t\nonumber\\
        &\leq
        4 \Bigl( ||\Sigma_0||_{L^2(\TT)}^2 + 1 \Bigr),
    \end{align}
    so that, using $\partial_x u = \frac{\Sigma + \Parti}{\mu}$,
    \begin{align}\label{Lemma 3: Better Bounds on the effective Pressure: Proof 9}
        ||\partial_x u(T,\cdot)||_{L^2(\TT)}^2
        \leq
        \frac{2}{\mu^2} ||\Sigma(T,\cdot)||_{L^2(\TT)}^2 + \frac{2}{\mu^2} ||\Parti||_{L^2(\TT)}^2
        \leq
        \frac{8}{\mu^2} \biggl( ||\Sigma_0||_{L^2(\TT)}^2 + 1 +|\Kparti|^2 \biggr).
    \end{align}
    Using $\eqref{Lemma 3: Better Bounds on the effective Pressure: Inserted}$ and $\eqref{Lemma 3: Better Bounds on the effective Pressure: Proof 7}$ in $\eqref{Lemma 3: Better Bounds on the effective Pressure: Proof 6}$, we conclude
    \begin{equation*}
        \begin{aligned}
            \frac{\mu}{4M_0}||\partial_x \Sigma||^2_{L^2(0,\tildeT;L^2(\TT)}
            \leq
            5 \Bigl(||\Sigma_0||^2_{L^2(\TT)} + 1 \Bigr).
        \end{aligned}
    \end{equation*}
    In total we obtain the inequality
    \begin{align*}
        ||\partial_x u||_{L^\infty(0,\tildeT;L^2(\TT))}^2 
        +
        ||\partial_x \Sigma||_{L^2(0,\tildeT;L^2(\TT))}^2
        \leq
        \frac{4}{\mu^2}\biggl(
        2 + 5 M_0 \mu
        \biggr)
        \biggl(||\Sigma_0||_{L^2(\TT)}^2 + 1 + |\Kparti|^2\biggr),
    \end{align*}    
    provided $\tildeT < T_u$. This yields the claim.
\end{proof}

We finish this section with the proof of Theorem~\ref{Thm:Refined A Priori Estimates}.

\begin{proof}[Proof of Theorem~\ref{Thm:Refined A Priori Estimates}]
    Using the strong regularity of $(\rho,u,c)$ there exists some $0<\tildeT<1$, such that the inequalities $\eqref{A Priori Estimates: Local-in-Time Bound density}$ and $\eqref{A Priori Estimate: Local-in-Time Bound effective pressure}$ hold on $[0,\tildeT]$.
    We define
    \begin{align*}
        T_0 := \min\bigl( 1, T_\rho, T_u \bigr),
    \end{align*}
    where $T_\rho, T_u \in (0,\infty)$ are the positive times provided by Lemma~\ref{Lemma1: Refined Velocity and Density Bound} and by Lemma~\ref{Lemma 3: Better Bound on the Effective Pressure}, respectively.
    First, we notice that Proposition~\ref{Prop:Energy Dissipation} yields
    \begin{align*}
        ||u||_{L^\infty(0,\tildeT;L^2(\TT))}^2 \leq 4M_0 E_0.
    \end{align*}
    Hence, by virtue of the bounds $\eqref{A Priori Estimates: Local-in-Time Bound density}$ and $\eqref{A Priori Estimate: Local-in-Time Bound effective pressure}$, we are done if we conclude $T_0 \leq \tildeT$. Let us suppose the contrary, i.e. $T_0 > \tildeT$. Then, we may assume without loss of generality that
    \begin{align}\label{Thm:Refined A Priori Estimates : Proof 1}
        \tildeT= \sup\{ T \in [0,T_0] \mid \eqref{A Priori Estimates: Local-in-Time Bound density}, \eqref{A Priori Estimate: Local-in-Time Bound effective pressure} \text{ hold on } [0,T]\} < T_0.
    \end{align}
    Since $\tildeT<T_0$, we have by Lemma~\ref{Lemma1: Refined Velocity and Density Bound} and by Lemma~\ref{Lemma 3: Better Bound on the Effective Pressure} that the sharper bounds $\eqref{Lemma 1: Refined Velocity and Density Bound: Result 1}$ and $\eqref{Lemma 3: Better Bound on the Effective Pressure:Result1}$ hold on $[0,\tildeT]$. Then, by continuity, we find some small $\varepsilon>0$, such that the bounds $\eqref{A Priori Estimates: Local-in-Time Bound density}$ and $\eqref{A Priori Estimate: Local-in-Time Bound effective pressure}$ hold on $[0,\tildeT + \varepsilon]$. This contradicts $\eqref{Thm:Refined A Priori Estimates : Proof 1}$.
\end{proof}


\section{Homogenization}\label{Homogenization}
In this section we perform the homogenization procedure and investigate the propagation of initial density oscillations for the system $\eqref{NSK Elliptic 1}$ with the usage of parametrized measures as in \cite{Hillairet_New_Physical}.
To do so, let us fix initial data $\rho^0_n,u^0_n \in H^1(\TT)$ for $n \in \NN$ satisfying the uniform bounds
\begin{align*}
    M_0^{-1} \leq \rho^0_n(x) \leq M_0, \quad \forall x \in \TT, \quad \forall n \in \NN
\end{align*}
and
\begin{align*}
    \sup\limits_{n\in\NN} ||u_n^0||_{H^1(\TT)}< \infty,
\end{align*}
where $M_0$ denotes some positive constant.
This setting allows for strong oscillations in the initial densities, for instance, we could choose the initial densities as
\begin{align*}
    \rho^0_n(x):= \rho^0(nx)
\end{align*}
for some density profile $\rho^0\in H^1(\TT)$ bounded from below by some positive constant.
In accordance with the hypotheses of Theorem~\ref{Thm:Main Result}, we may use the Banach--Alaoglu theorem, to obtain some $u \in H^1(\TT)$, such that after passing to a subsequence
\begin{align*}
    u^0_n \weak u_0 \quad \text{in } H^1(\TT).
\end{align*}
Then, according to the well-posedness result Theorem~\ref{Thm:Global Existence and Uniqueness}, we construct for $n \in \NN$ the global-in-time strong solution $(\rho_n, u_n,c_n)$ of $\eqref{NSK Elliptic 1},\eqref{NSK Elliptic initial condition}$ with initial data $(\rho_n^0,u_n^0)$.
Let us denote as in Section~\ref{Refined A-Priori Estimates} the corresponding effective viscous flux as
\begin{align*}
    \Sigma_n := \mu \partial_x u_n + \Parti(\rho_n) \quad \forall \, n \in \NN.
\end{align*}
The a-priori bounds provided by Proposition~\ref{Prop:Energy Dissipation}, Proposition~\ref{prop:square control on density} and Theorem~\ref{Thm:Refined A Priori Estimates} then imply the following uniform bounds and convergences.

\begin{lemma}\label{lem:Uniform Bound on the Solutions}
Let the hypotheses of Theorem~\ref{Thm:Main Result} hold true.\\
Then there exists some $T_0>0$, such that we have the following uniform bounds:
\begin{enumerate}[1.]
    \item \label{unif_bound_rho} $\rho_n, \frac{1}{\rho_n}, \Parti(\rho_n)$ are uniformly bounded in $L^\infty(0,T_0;L^\infty(\TT))$, in particular
    \begin{align*}
        (2M_0)^{-1} \leq \rho_n(t,x) \leq 2M_0 \quad \forall \, (t,x) \in [0,T_0] \times \TT, \quad \forall \, n \in \NN,
    \end{align*}
    \item \label{unif_bound_u} $u_n$ is uniformly bounded in $L^\infty(0,T_0;H^1(\TT))$,
    \item \label{unif_bound_partial_u} $\partial_x u_n$ is uniformly bounded in $L^2(0,T_0;L^\infty(\TT))$,
    \item \label{unif_bound_sigma} $\Sigma_n$ is uniformly bounded in $L^2(0,T_0;H^1(\TT))$,
    \item \label{unif_bound_c} $c_n$ is uniformly bounded in $L^\infty(0,T_0;H^2(\TT))$.
\end{enumerate}
Moreover, there exists $\rho, \overline{P}\in L^\infty(0,T_0;L^\infty(\TT))$, $u \in L^\infty(0,T_0;H^1(\TT))$ with $\partial_x u \in L^2(0,T_0;L^\infty(\TT))$, $\Sigma \in L^2(0,T_0;H^1(\TT))$ and $c \in L^\infty(0,T_0;H^2(\TT))$, such that, after passing to a subsequence, we have the following convergences:
    \begin{enumerate}[1.]
    \setcounter{enumi}{5}
        \item \label{conv_rho_pressure} $\rho_n,\Parti(\rho_n) \weakstar \rho,\overline{P}$ in $L^\infty(0,T_0;L^\infty(\TT))$,
        \item  \label{conv_u} $u_n \weakstar u$ in $L^\infty(0,T_0;H^1(\TT))$ and $\partial_x u \in L^2(0,T_0;L^\infty(\TT))$,
        \item \label{conv_sigma} $\Sigma_n \weak \Sigma$ in $L^2(0,T_0;H^1(\TT))$,
        \item \label{conv_c} $c_n \weakstar c$ in $L^\infty(0,T_0;H^2(\TT))$.
    \end{enumerate}
Furthermore, we have
\begin{align}\label{Prop:Compactness Quantities : Result 1}
        (2M_0)^{-1} \leq \rho(t,x) \leq 2M_0 \quad \forall \, (t,x) \in [0,T_0] \times \TT,
    \end{align}
    and
    \begin{align}\label{Prop:Compactness Quantities : Result 2}
        ||u||_{L^\infty(0,T_0;H^1(\TT))}^2 
        +
        ||\Sigma||_{L^2(0,T_0;H^1(\TT))}^2 
        \leq
        C_1,
    \end{align}
    where $C_1$ is a positive constant that only depends on $M_0, ||u_0||_{H^1(\TT)},\mu, \coup, \kappa$.
\end{lemma}
\begin{proof}
    The uniform bounds \ref{unif_bound_rho}, \ref{unif_bound_u}, \ref{unif_bound_partial_u}, \ref{unif_bound_sigma} and \ref{unif_bound_c} follow from Proposition~\ref{Prop:Energy Dissipation}, Proposition~\ref{prop:square control on density} and Theorem~\ref{Thm:Refined A Priori Estimates} by using the uniform bounds on $\rho^0_n$ and $u^0_n$ according to the hypotheses of Theorem~\ref{Thm:Main Result}.
    These bounds imply via the Banach--Alaoglu theorem, after passing to a subsequence, the convergences \ref{conv_rho_pressure}, \ref{conv_u}, \ref{conv_sigma} and \ref{conv_c}.
    The regularity $\partial_x u \in L^2(0,T_0;L^\infty(\TT))$ and the inequalities $\eqref{Prop:Compactness Quantities : Result 1}$ and $\eqref{Prop:Compactness Quantities : Result 2}$ follow from semi-continuity results on weak convergence.
\end{proof}
Using the Aubin--Lions lemma (see for instance \cite{Simon}), we obtain stronger convergence for the velocity.

\begin{lemma}[Improved Convergences]\label{lem: Improved Convergences}
    Under the hypotheses and notation of Lemma~\ref{lem:Uniform Bound on the Solutions}, we have that
    \begin{enumerate}[1.]
        \item \label{conv_strong_u} $u_n \to u$ in $C([0,T_0];C(\TT))$,
        \item \label{conv_u_square} $|u_n|^2 \to |u|$ in $L^2(0,T_0;L^2(\TT))$,
        \item \label{conv_rho_u} $\rho_n u_n \weak \rho u$ in $L^2(0,T_0;L^2(\TT))$,
        \item \label{conv_rho_u_square} $\rho_n |u_n|^2 \weak \rho |u|^2$ in $L^2(0,T_0;L^2(\TT))$. 
    \end{enumerate}
\end{lemma}
\begin{proof}
    With \ref{unif_bound_rho} in Lemma~\ref{lem:Uniform Bound on the Solutions}, we may rewrite the momentum equation $\eqref{NSK Elliptic 1}_2$ using the continuity equation $\eqref{NSK Elliptic 1}_1$ as
    \begin{align*}
        \partial_t u_n = -u_n\partial_x u_n + \frac{1}{\rho_n} \partial_x \Sigma_n + \coup \partial_x c_n .
    \end{align*}
    This yields
    \begin{align*}
        &||\partial_t u_n||_{L^2(0,T_0;L^2(\TT))}\\
        &\leq
        ||u_n||_{L^\infty(0,T_0;L^\infty(\TT))}||u_n||_{L^2(0,T_0;H^1(\TT))}
        +
        \biggl|\biggl|\frac{1}{\rho_n}\biggr|\biggr|_{L^\infty(0,T_0;L^\infty(\TT))} ||\Sigma_n||_{L^2(0,T_0;H^1(\TT))}\\
        &\qquad+
        \coup
        ||c_n||_{L^\infty(0,T_0;H^2(\TT))}.
    \end{align*}
    By virtue of the uniform bounds in Lemma~\ref{lem:Uniform Bound on the Solutions} the above inequality implies that $\partial_t u_n$ is uniformly bounded in $L^2(0,T_0;L^2(\TT))$. Together with \ref{unif_bound_u} in Lemma~\ref{lem:Uniform Bound on the Solutions} and the fact that the Sobolev embedding $H^1(\TT) \hookrightarrow C(\TT)$ is compact, we may apply the Aubin--Lion lemma (see for instance \cite{Simon}) to conclude that, after passing to a subsequence,
    \begin{align}\label{convergence subsequence}
        u_n \to u \quad \text{in } C([0,T_0];C(\TT)).
    \end{align}
    Since we already know the convergence \ref{conv_u} in Lemma~\ref{lem:Uniform Bound on the Solutions} for the whole sequence, we conclude that $\eqref{convergence subsequence}$ holds in fact without passing to a subsequence.
    The convergence \ref{conv_strong_u} implies clearly \ref{conv_u_square}. 
    Combining the strong convergence of $u_n$ and $|u_n|^2$ with the weak convergence of $\rho_n$ yields the convergences \ref{conv_rho_u} and \ref{conv_rho_u_square}.
\end{proof}

Before passing to the limit in system $\eqref{NSK Elliptic 1}$, we need one more compactness result concerning the effective viscous flux, that was one of the key ingredients in the existence proof of global finite energy weak solutions (\cite{Lions}, \cite{Feireisl}). This compactness result was transferred to the one-dimensional periodic framework in \cite{Hillairet_New_Physical}.

\begin{lemma}\label{lem: Compensated Compactness}
Let the hypotheses and notation of Lemma~\ref{lem:Uniform Bound on the Solutions} hold true and let $\beta \in C^1(\RR)$.
Then there exists a function $\overline{\beta} \in L^\infty(0,T_0;L^\infty(\TT))$ such that, after passing to a subsequence, we have
    \begin{enumerate}[1.]
        \item \label{conv_beta} $\beta(\rho_n) \weakstar \overline{\beta}$ in $L^\infty(0,T_0;L^\infty(\TT))$,
        \item \label{conv_beta_sigma}$\beta(\rho_n)\Sigma_n \weak \overline{\beta} \Sigma$ in $L^2(0,T_0;L^2(\TT))$,
        \item \label{conv_beta_c}$\beta(\rho_n) \partial_x c_n \weak \overline{\beta} \partial_x c$ in $L^2(0,T_0;L^2(\TT))$. 
    \end{enumerate}
\end{lemma}
\begin{proof}
    For the proof of \ref{conv_beta} and \ref{conv_beta_sigma} we refer to Lemma 10 in \cite{Hillairet_New_Physical}. 
    As $\partial_x c_n$ is uniformly bounded in $L^2(0,T_0;H^1(\TT))$, the proof for \ref{conv_beta_sigma} also applies for \ref{conv_beta_c}.
\end{proof}

Finally, we pass to the limit in system $\eqref{NSK Elliptic 1}$. We obtain the following system for the limit quantities $(\rho,u,c)$:

\begin{proposition}[Limit system]\label{Prop:Limit System}
    Under the hypotheses and notation of Lemma~\ref{lem:Uniform Bound on the Solutions}, we have that $(\rho,u,c)$ solves the system
    \begin{alignat}{2}\label{Prop:Limit System : Result 1}
    \left\{
        \begin{aligned}
            \partial_t \rho + \partial_x(\rho u) &= 0,\\
            \partial_t (\rho u) + \partial_x(\rho u^2)
            - \partial_x \Sigma - \coup \rho \partial_x c
            &= 0,\\
            -\kappa \partial_{xx} c + \coup (c-\rho) &= 0
        \end{aligned}
    \right.
    \end{alignat}
    in $\curlyD^\prime((0,T_0)\times\TT)$, where
    \begin{align}\label{Prop:Limit System : Result 2}
        \Sigma = \mu \partial_x u - \overline{P} \quad \text{a.e. in } (0,T_0) \times \TT.
    \end{align}
    Moreover, we have
    \begin{align}\label{Prop:Limit System : Result 3}
        u(0,\cdot) = u_0 \quad \text{a.e. in } \TT.
    \end{align}
\end{proposition}
\begin{proof}
    The fact that $(\rho,u,c)$ solves $\eqref{Prop:Limit System : Result 1}$ in $\mathcal{D}^\prime((0,T_0) \times \TT)$ follows from Lemma~\ref{lem:Uniform Bound on the Solutions}, Lemma~\ref{lem: Improved Convergences} and Lemma~\ref{lem: Compensated Compactness} with $\beta = \mathrm{Id}_\RR$.
    In particular, Lemma~\ref{lem:Uniform Bound on the Solutions} implies that both sides of
    \begin{align*}
        \Sigma_n = \mu \partial_x u_n - \Parti(\rho_n)
    \end{align*}
    converge weakly in $L^2(0,T_0;L^2(\TT))$ to their corresponding limits. This implies $\eqref{Prop:Limit System : Result 2}$.
    Assertion ~\eqref{Prop:Limit System : Result 3} follows from the convergence \ref{conv_strong_u} in Lemma~\ref{lem: Improved Convergences}.
\end{proof}

We notice that we have obtained an unclosed quantity $\overline{P}$. Due to the strong oscillations in the density, we cannot expect any strong convergence of the density. In particular, we cannot hope for a relation $\overline{P} = \Parti(\rho)$ due to the nonlinear nature of the pressure function $\Parti$. In fact, this is not what we want to achieve, since this would imply that all oscillations in the density sequence $(\rho_n)_{n\in\NN}$ would have been disappeared.
In order to close the system mathematically, we interpret the density sequence $(\rho_n)_{n\in\NN}$ as a sequence of parametrized measures $(\Theta_n)_{n\in\NN}$ and investigate the limit of this sequence.
More precisely, we define for each $t \in [0,T_0]$ the map
\begin{align}\label{defi:measure:n}
    C^0_c((0,T_0)\times \TT) \to \RR,\quad 
    \langle \Theta_n(t), \varphi \rangle 
    := \TTint \varphi(x,\rho_n(t,x)) \,\dd x
\end{align}
It is not hard to see, that for each $t \in [0,T_0]$ we have $\Theta_n(t) \in \curlyM^+(\TT\times \RR)$.
In the following proposition, we collect some properties of this map, including the fact that we can extract a converging subsequence of $(\Theta_n)_{n\in\NN}$ and identify a kinetic equation for the limit measure. 
This equation can be seen as a closure for the system $\eqref{Prop:Limit System : Result 1}$.

\begin{proposition}[Properties of $\Theta_n$]\label{prop: Properties of Theta(n)}
    Under the hypotheses and notation of Lemma~\ref{lem:Uniform Bound on the Solutions}, let $\Theta_n$ be defined via $\eqref{defi:measure:n}$ for $n \in \NN$.\\
    Then we have $\Theta_n \in \Cw([0,T_0];\curlyM^+(\TT \times \RR))$. 
    Furthermore, we have for any $t \in [0,T_0]$
    \begin{align}\label{prop: Properties of Theta(n) : Result 1}
        \spt\, \bigl(\Theta_n(t)\bigr) \subseteq \TT \times \bigl[(2M_0)^{-1}, 2M_0 \bigr]
    \end{align}
    and
    \begin{align}\label{prop: Properties of Theta(n) : Result 2}
        ||\Theta_n(t)||_{\curlyM^+(\TT \times \RR)} \leq 1.
    \end{align}
    Moreover, there exists some $\Theta \in \Cw([0,T_0];\curlyM^+(\RR \times \TT))$, such that, after passing to a subsequence,
    \begin{align}\label{prop: Properties of Theta(n) : Result 3}
        \Theta_n \to \Theta \quad \text{in } \Cw([0,T_0];\curlyM^+(\TT \times \RR)),
    \end{align}
    that is,
    \begin{align*}
        \sup\limits_{t \in [0,T_0]}\bigl|\langle \Theta_n(t)- \Theta(t), \varphi \rangle \bigr| \to 0
        \quad \forall \, \varphi \in C^0_0(\TT \times \RR).
    \end{align*}
    Furthermore, $\Theta$ satisfies
    \begin{align}\label{prop: Properties of Theta(n) : Result 4}
        \spt\,\bigl(\Theta(t)\bigr) \subseteq \TT \times \bigl[(2M_0)^{-1}, 2M_0\bigr]
    \end{align}
    and
    \begin{align}\label{prop: Properties of Theta(n) : Result 5}
        \partial_t \Theta 
        +
        \partial_x(\Theta u)
        -
        \frac{1}{\mu}\partial_\xi \biggl( [\xi \Sigma + \xi \Parti(\xi)] \Theta \biggr)
        -
        \frac{1}{\mu}\biggl( [\Sigma + \Parti(\xi)] \Theta \biggr)
        =
        0
    \end{align}
    in $\curlyD^\prime((0,T_0) \times \TT \times \RR)$.
\end{proposition}
\begin{proof}
    For $\varphi \in C^0_c(\TT \times \RR)$, we have
    \begin{align*}
        \langle \Theta_n(t),\varphi \rangle 
        = \TTint \varphi(x,\rho_n(t,x)) \dd x 
        \leq
        ||\varphi||_{L^\infty(\TT \times \RR)},
    \end{align*}
    and therefore $\eqref{prop: Properties of Theta(n) : Result 2}$. 
    For the weak continuity, we fix $t,s \in [0,T_0]$ and verify with the mean value theorem for $\varphi \in C^1_c(\TT \times \RR)$
    \begin{align*}
        \langle \Theta_n(t) - \Theta_n(s), \varphi \rangle 
        \leq
        ||\partial_2\varphi||_{L^\infty(\TT \times \RR)} ||\rho_n(t,\cdot) - \rho_n(s,\cdot)||_{L^1(\TT)}.
    \end{align*}
    Here $\partial_2$ denotes the partial derivative with respect to the second variable.
    Recalling that $\rho_n \in C([0,T_0];H^1(\TT))$ and using the fact that $C^\infty_c(\TT \times \RR)$ lies dense in $C^0_0(\TT \times \RR)$ with respect to $||\cdot||_{L^\infty(\TT \times \RR)}$ yields the continuity of the map
    \begin{align*}
      [0,T_0] \ni t \mapsto \langle \Theta_n(t), \varphi \rangle \in \RR
    \end{align*}
    for any $\varphi \in C^0_c(\TT \times \RR)$. Hence $\Theta_n \in \Cw([0,T_0]; \curlyM^+(\TT \times \RR))$. 
    To verify $\eqref{prop: Properties of Theta(n) : Result 1}$, we take $\varphi \in C^0_c(\TT \times \RR)$ with $\spt\, (\varphi) \subseteq \TT \times \bigl(\RR \setminus [(2M_0)^{-1},2M_0]\bigr)$. Then we have for any $t \in [0,T_0]$
    \begin{align*}
        \langle \Theta_n(t), \varphi \rangle 
        = \TTint \varphi(x,\rho_n(t,x)) \dd x = 0,
    \end{align*}
    by virtue of \ref{unif_bound_rho} in Lemma~\ref{lem:Uniform Bound on the Solutions}. Thus assertion $\eqref{prop: Properties of Theta(n) : Result 1}$ follows.
    The convergence of $(\Theta_n)_{n\in\NN}$ follows from an application of the Arzelà--Ascoli theorem by a density argument.
    The fact that the limit measure $\Theta$ satisfies $\eqref{prop: Properties of Theta(n) : Result 5}$ follows from Lemma~\ref{lem:Uniform Bound on the Solutions}, Lemma~\ref{lem: Improved Convergences} and Lemma~\ref{lem: Compensated Compactness}.
    For a detailed proof we refer the reader to the proof of Proposition 12 in \cite{Hillairet_New_Physical}.
    Relation $\eqref{prop: Properties of Theta(n) : Result 4}$ is a consequence of the convergence $\eqref{prop: Properties of Theta(n) : Result 3}$ and $\eqref{prop: Properties of Theta(n) : Result 1}$.
\end{proof}


Following \cite{Hill_Note}, we construct now an appropriate solution to the target BN system $\eqref{BN-System to justify}$ subject to the initial conditions $\alpha_+^0,\alpha_-^0,\rho_+^0,\rho_-^0$ given by the hypothesis of Theorem~\ref{Thm:Main Result}.
This solution gives then rise to a parametrized measure that satisfy the same equation $\eqref{prop: Properties of Theta(n) : Result 5}$ as $\Theta$ with the same initial condition (see hypothesis of Theorem~\ref{Thm:Main Result}). 
Then, a uniqueness result concludes the proof of Theorem~\ref{Thm:Main Result}.
Since the techniques used to prove the following theorem follow techniques demonstrated in \cite{Hill_Note}, we give the proof of the following result in the Appendix.

\begin{theorem}\label{thm: Existence Solutions to the BN-System}
    Let the hypotheses and notation of Lemma~\ref{lem:Uniform Bound on the Solutions} hold true.
    Then there exist some $T_1 \in (0,T_0]$ and a weak solution
    \begin{align*}
        \rho_{\pm},\alpha_{\pm} \in L^\infty(0,T_1;L^\infty(\TT)) \cap C([0,T_1];L^1(\TT))
    \end{align*}
    of the BN system
    \begin{alignat}{2}\label{thm: Existence Solutions to the BN-system : Result 2}
    \left\{
        \begin{aligned}
            \partial_t \alpha_\pm + u \partial_x \alpha_{\pm} 
            &= \frac{\alpha_\pm }{\mu} (\Parti(\rho_\pm) - \overline{P})
            \quad &&\text{in } (0,T_1) \times \TT,\\
            \partial_t \rho_{\pm} + \partial_x(\rho_\pm u)
            &=
            \frac{\rho_\pm}{\mu} (\overline{P} - \Parti(\rho_\pm))
            \quad &&\text{in } (0,T_1) \times \TT,\\
            \rho_\pm(0,\cdot) &= \rho_\pm^0, \quad
            \alpha_\pm(0,\cdot) = \alpha_\pm^0 \quad
            &&\text{in } \TT,
        \end{aligned}
        \right.
    \end{alignat}
    where $u \in L^\infty(0,T_0;H^1(\TT))$ with $\partial_x u \in L^2(0,T_0;L^\infty(\TT))$ denotes the limit velocity in Lemma~\ref{lem:Uniform Bound on the Solutions}.\\
    Moreover, we have
    \begin{align}\label{thm: Existence Solutions to the BN-system : Result 3}
       \alpha_\pm \geq 0, \quad 
        (2M_0)^{-1} \leq \rho_\pm \leq 2M_0 
        \quad \text{a.e. on } [0,T_1] \times \TT.
    \end{align}
\end{theorem}
\begin{proof}
    This result follows directly from Theorem~\ref{thm:App:Existence BN Solutions}.
\end{proof}

The constructed solutions to the BN system gives rise to a parametrized measure $\othet$ as follows:\\
For $t \in [0,T_1]$, we define $\othet(t)$ via
\begin{align}\label{defi:othet}
    \langle \othet(t),\varphi\rangle := 
    \TTint \alpha_+(t,x) \varphi(x,\rho_+(t,x)) 
    +
    \alpha_-(t,x) \varphi(x,\rho_-(t,x)) \, \mathrm{d}x
\end{align}
for $\varphi \in C^0_c(\TT \times \RR)$.
It is straightforward to verify that $\othet(t) \in \curlyM^+(\TT \times \RR)$.
In the following proposition, we collect some properties of $\othet$, including the crucial fact that $\othet$ satisfies the same equation as $\Theta$.

\begin{proposition}[Properties of $\othet$]\label{prop: Properties of othet}
    Under the hypothesis and notation of Theorem~\ref{thm: Existence Solutions to the BN-System}, let us define $\othet$ via $\eqref{defi:othet}$.\\
    Then we have $\othet \in \Cw([0,T_1];\curlyM^+(\TT \times \RR))$. Furthermore, for any $t \in [0,T_1]$, we have
    \begin{align}\label{prop: Properties of othet : Result 1}
        \spt\, \bigl(\othet(t)\bigr) \subseteq \TT \times \bigl[(2M_0)^{-1},2M_0\bigr]
    \end{align}
    and
    \begin{align}\label{prop: Properties of othet : Result 2}
        ||\othet(t)||_{\curlyM(\TT \times \RR)} 
        \leq 
        ||\alpha_+||_{C([0,T_1];L^1(\TT))} + ||\alpha_-||_{C([0,T_1];L^1(\TT))}.
    \end{align}
    Moreover, $\othet$ satisfies 
    \begin{align}\label{prop: Properties of othet : Result 3}
        \partial_t \othet
        +
        \partial_x(\othet u)
        -
        \frac{1}{\mu}\partial_\xi\biggl( [\xi \Sigma + \xi \Parti(\xi)]\othet \biggr)
        -
        \frac{1}{\mu}\biggl( [\Sigma + \Parti(\xi)] \othet \biggr) 
        = 0
    \end{align}
    in $\curlyD^\prime((0,T_1) \times \TT \times \RR)$.
\end{proposition}
\begin{proof}
    The weak continuity of $\othet$ follows from the continuity of $\alpha_\pm$ and $\rho_\pm$.
    Assertion $\eqref{prop: Properties of othet : Result 1}$ follows from $\eqref{thm: Existence Solutions to the BN-system : Result 3}$ and assertion $\eqref{prop: Properties of othet : Result 2}$ is trivial.
    Finally, equation $\eqref{prop: Properties of othet : Result 3}$ follows from a regularization argument (see for instance \cite{DiPerna}) and the fact that $\Sigma = \mu \partial_x u - \overline{P}$.
\end{proof}

Finally, using the properties of $\Theta$ and $\othet$ and a uniqueness result for measure valued solutions to transport equations (see Theorem~\ref{thm:App:Uniqueness}), we prove our main result Theorem~\ref{Thm:Main Result}.


\begin{proof}[Proof of Theorem~\ref{Thm:Main Result}]
    We stick to the notation of Proposition~\ref{prop: Properties of Theta(n)} and Proposition~\ref{prop: Properties of othet}.
    We investigate the difference $\nu := \Theta - \othet$. We have $\nu \in \Cw([0,T_1];\curlyM(\TT \times \RR))$ with 
    \begin{align*}
        \sup\limits_{t\in[0,T_0]}||\nu(t)||_{\curlyM(\TT \times \RR)} < \infty
    \end{align*}
    and 
    \begin{align}\label{Thm:Main Result : Proof 1}
        \spt\,\bigl(\nu(t)\bigr) \subseteq \TT \times \bigl[(2M_0)^{-1}, 2M_0\bigr] \quad \forall \, t \in [0,T_1].
    \end{align}
    Furthermore, we have by hypotheses of Theorem~\ref{Thm:Main Result} that $\Theta(0) = \othet(0)$, so that
    \begin{align*}
        \nu(0) = 0.
    \end{align*}
    Moreover, $\nu$ satisfies
    \begin{align*}
        \partial_t \nu 
        +
        \partial_x( \nu u)
        -
        \frac{1}{\mu}\partial_\xi\biggl( (\xi \Sigma - \xi \Parti(\xi)) \nu \biggr)
        -
        \frac{1}{\mu}\biggl( (\Sigma - \Parti(\xi)) \nu \biggr)
        = 0
    \end{align*}
    in $\curlyD^\prime((0,T_0) \times \TT \times \RR)$
    since both, $\Theta$ and $\othet$ satisfy this equation (see Proposition~\ref{prop: Properties of Theta(n)} and Proposition~\ref{prop: Properties of othet}).
    Now, we choose a cutoff function $\chi \in C^\infty_c(\RR)$ such that
    \begin{align*}
        \chi \equiv 1 \quad \text{on } \bigl[(2M_0)^{-1},2M_0\bigr].
    \end{align*}
    By virtue of $\eqref{Thm:Main Result : Proof 1}$ we conclude with the definition of $\chi$, that $\nu$ satisfies in fact
    \begin{align*}
        \partial_t \nu
        +
        \mathrm{div}_{(x,\xi)}(\mathbf{V}\nu) 
        +
        g\nu
        =
        0
        \quad \text{in } \curlyD^\prime((0,T_0) \times \TT \times \RR),
    \end{align*}
    with $\mathbf{V}:=(V_1,V_2)$,
    \begin{align*}
        V_1(t,x,\xi):= u(t,x),
        \quad
        V_2(t,x,\xi):= -\frac{1}{\mu}(\xi\Sigma(t,x)-\xi\Parti(\xi))\chi(\xi)
    \end{align*}
    and
    \begin{align*}
        g(t,x,\xi):=-\frac{1}{\mu}(\Sigma(t,x) - \Parti(\xi))\chi(\xi).
    \end{align*}
    We notice that all prerequisites of Theorem~\ref{thm:App:Uniqueness} are satisfied, such that we can conclude
    \begin{align*}
        \nu = 0,
    \end{align*}
    i.e. $\Theta = \othet$.
    By the convergence \ref{conv_rho_pressure} in Lemma~\ref{lem:Uniform Bound on the Solutions}, we have for any $\eta\in C^\infty_c(0,T_1)$ and any $\phi \in C^\infty(\TT)$, that 
    \begin{align*}
        \int_0^{T_1}\TTint \Parti(\rho_n(t,x)) \eta(t) \phi(x) \; \dd x\, \dd t
        \longrightarrow 
        \int_0^{T_1} \TTint \overline{P}(t,x) \eta(t)\phi(x) \, \dd x \, \dd t.
    \end{align*}
    On the other hand, we have since $\othet = \Theta$ that
    \begin{align*}
        &\int_0^{T_1} \TTint \Parti(\rho_n(t,x)) \eta(t) \phi(x) \, \dd x \, \dd t
        = 
        \int_0^{T_1} \eta(t) \langle \Theta_n(t),\Parti(\xi) \phi(x)\rangle \, \dd t\\
        &\longrightarrow
        \int_0^{T_1} \eta(t) \langle \othet(t),\Parti(\xi)\phi(x)\rangle \, \dd t\\
        &\qquad=\int_0^{T_1} \TTint \eta(t) \phi(x) \bigl( \alpha_+(t,x)\Parti(\rho_+(t,x)) + \alpha_-(t,x)\Parti(\rho_-(t,x))  \bigr) \, \dd x \, \dd t.
    \end{align*}
    Since $\eta$ and $\phi$ were arbitrary, we conclude
    \begin{align}\label{Final Proof:eq1}
        \overline{P} = \alpha_+ \Parti(\rho_+) + \alpha_- \Parti(\rho_-).
    \end{align}
    By similar arguments, we conclude
    \begin{align}
        \rho = \alpha_+ \rho_+ + \alpha_- \rho_-.
    \end{align}
    Using the relations $\eqref{Final Proof:eq1}$ in $\eqref{thm: Existence Solutions to the BN-system : Result 2}_1$ and $\eqref{thm: Existence Solutions to the BN-system : Result 2}_2$ and adding up both equations yields that $\alpha_+ + \alpha_-$ satisfies
    \begin{alignat*}{2}
    \left\{
        \begin{aligned}
            \partial_t(\alpha_+ + \alpha_-) + u \partial_x(\alpha_+ + \alpha_-) &= \frac{\overline{P} - (\alpha_+ + \alpha_-) \overline{P}}{\mu} 
            \quad &&\text{in } (0,T_1) \times \TT,\\
            (\alpha_+ + \alpha_-)(0,\cdot) &= 1 \quad &&\text{in } \TT.
        \end{aligned}
    \right.
    \end{alignat*}
    By uniqueness (see e.g. \cite{DiPerna}), we conclude that 
    \begin{align}\label{Final proof:eq3}
        \alpha_+ + \alpha_- = 1 \quad \text{a.e. on } [0,T_1] \times \RR.
    \end{align}
    Using the relations $\eqref{Final Proof:eq1}$ and $\eqref{Final proof:eq3}$ in $\eqref{thm: Existence Solutions to the BN-system : Result 2}$ yields that $(\alpha_+,\alpha_-,\rho_+,\rho_-,u,c)$ solves $\eqref{BN-System to justify}$ in $\mathcal{D}^\prime((0,T_1)\times\TT)$. 
    The proof of Theorem~\ref{Thm:Main Result} is now complete.
    
\end{proof}


\section{Conclusions}\label{Conclusions}
In this paper, we have investigated the propagation of initial density oscillations for the non-local NSK system.
With the use of parametrized measures, we have derived a closed homogenized system consisting of a momentum equation for the velocity and a kinetic equation for a parametrized measure.
After assuming that the parametrized measure is a convex combination of Dirac-measures initially, we have proven that the kinetic equation preserves this structure.
With that structure for the parametrized measure, the kinetic equation then reduces to the BN system $\eqref{BN-System to justify}$. 
In that sense, we have justified the BN system $\eqref{BN-System to justify}$ rigorously as macroscopic description for a compressible liquid-vapor flow that is modeled with the non-local NSK equations on the detailed scale.
It would be interesting to extend this work to the 3D case as in \cite{HillProp} in the framework of finite-energy weak solutions.
It seems that the arguments of \cite{HillProp} should apply in order to prove that the effective equations are given by a kinetic equation for a parametrized measure and a momentum equation for the velocity.
However, proving the propagation of convex combinations of Dirac-measures seems difficult, since the arguments in \cite{Hill_Note} rely on an isentropic pressure law. 
As mentioned by the authors in \cite{Hill_Note}, it is not clear how these arguments carry over to a more general pressure law as for instance a pressure law of Van-der-Waals type.
Nevertheless, we could interpret the parametrized measure as a probability density function as done in \cite{Plotnikov} and investigate the resulting equations numerically.

\appendix


\section{Existence of Solutions to a BN System}\label{Existence of Solutions to a Baer--Nunziato System}

Here, we provide a proof of the following result that is concerned with the existence of weak solutions to a BN system. 
The proof follows the lines in \cite{Hill_Note}. 
There, the author proved the analogous result for an isentropic pressure law.
Since we are dealing with a different pressure law, we have to adjust the proof.
The precise statement reads as follows:

\begin{theorem}\label{thm:App:Existence BN Solutions}
    Let $T_0>0$ and $\rho_0, \alpha_0 \in L^\infty(\TT)$ with
        \begin{align*}
            \alpha_0(x)\geq 0, \quad M_0^{-1} \leq \rho_0(x) \leq M_0 \quad \text{for a.e. } x \in \TT
        \end{align*}
        for some positive constant $M_0$.
        Assume that $(P,\gamma)$ is admissible and let $\pi \in L^\infty(0,T_0;L^\infty(\TT))$ and $u \in L^\infty(0,T_0;L^2(\TT)) \cap L^2(0,T_0;H^1(\TT))$ with $\partial_x u \in L^1(0,T_0;L^\infty(\TT))$.\\
        Then there exist some time $T_1 \in (0,T_0]$ and $\alpha,\rho \in L^\infty(0,T_1;L^\infty(\TT)) \cap C([0,T_1];L^1(\TT))$, such that
        \begin{alignat}{2}\label{thm:App:Existence BN Solutions : Result 1}
        \left\{
            \begin{aligned}
                    \partial_t \alpha + u \partial_x \alpha &= \alpha(\Parti(\rho) - \pi),\\
                    \partial_t \rho + \partial_x(\rho u) &= \rho(\pi - \Parti(\rho))
            \end{aligned}
        \right.
        \end{alignat}
        hold in $\curlyD^\prime((0,T_1)\times \TT)$ and 
        \begin{align}\label{thm:App:Existence BN Solutions : Result 3}
            \alpha(0,\cdot) = \alpha_0,\quad \rho(0,\cdot) = \rho_0 \quad \text{a.e. in  } \TT.
        \end{align}
        Moreover, the solution $(\alpha,\rho)$ satisfies
        \begin{align}\label{thm:App:Existence BN Solutions : Result 2}
            \alpha(t,x)\geq 0, \quad (2M_0)^{-1} \leq \rho(t,x) \leq 2M_0 
            \quad \text{for a.e. } (t,x) \in [0,T_1] \times \TT.
        \end{align}
\end{theorem}

In order to prove this theorem, we first investigate the linearized equation corresponding to $\eqref{thm:App:Existence BN Solutions : Result 1}$. The following result follows from the results in \cite{DiPerna}.

\begin{lemma}\label{lem:App: Existence linearized BN}
    Let $T_0>0$ and $\alpha_0,\rho_0 \in L^\infty(\TT)$. Let $u \in L^\infty(0,T_0;L^\infty(\TT))\cap L^2(0,T_0;H^1(\TT))$ with $\partial_x u \in L^1(0,T_0;L^\infty(\TT))$.\\ 
    Then, for any $f,g \in L^\infty(0,T_0;L^\infty(\TT))$, there exist unique functions
    \begin{align*}
        \alpha,\rho \in L^\infty(0,T_0;L^\infty(\TT))\cap C([0,T_0];L^1(\TT))
    \end{align*}
    that satisfy
    \begin{alignat}{2}\label{lem:App: Existence linearized BN : Result 1}
    \left\{
        \begin{aligned}
            \partial_t\alpha + u \partial_x \alpha &= \alpha f,\\ 
            \partial_t \rho + \partial_x(\rho u) &= \rho g
        \end{aligned}
    \right.
    \end{alignat}
    in $\curlyD^\prime((0,T_0)\times\TT)$ and
    \begin{align}\label{lem:App: Existence linearized BN : Result 2}
        \alpha(0,\cdot) = \alpha_0, \quad \rho(0,\cdot) = \rho_0 \quad \text{a.e. in } \TT.
    \end{align}
\end{lemma}
\begin{proof}
    See Proposition II.1 and Corollary II.2 in \cite{DiPerna}.
\end{proof}

Lemma~\ref{lem:App: Existence linearized BN} gives rise to a solution operator for any $T \in (0,T_0]$ via
    \begin{align*}
        \curlyL_{T}\colon L^\infty(0,T;L^\infty(\TT))^2 \to L^\infty(0,T;L^\infty(\TT))^2,
        \quad (f,g) \mapsto (\alpha,\rho),
    \end{align*}
    where $\alpha,\rho$ is the unique weak solution to problem $\eqref{lem:App: Existence linearized BN : Result 1}, \eqref{lem:App: Existence linearized BN : Result 2}$.
    This solution operator has the following properties:

\begin{lemma}\label{lem:App: Linear Solution Operator Properties}
        Under the assumptions of Lemma~\ref{lem:App: Existence linearized BN}, there exists some $C_0>0$ only depending on
        \begin{align*}
            ||\alpha_0||_{L^\infty(\TT)}, ||\rho_0||_{L^\infty(\TT)}, ||\partial_xu||_{L^1(0,T_0;L^\infty(\TT))},
        \end{align*}
        such that for any $M>0$ and any $T\in(0,T_0]$, we have 
        \begin{align}\label{lem:App: Linear Solution Operator Properties : Result 1}
            \curlyL_{T}( B(0,M) ) \subseteq B(0, C_0 \exp(T M) ),
        \end{align}
        where $B(0,M) \subseteq L^\infty(0,T;L^\infty(\TT))^2$ denotes the ball of radius $M$ in $L^\infty(0,T;L^\infty(\TT))^2$.\\
        Moreover, if $T \leq \min(T_0,1)$, we have for any $(\overline{f},\overline{g}), (\hat{f},\hat{g}) \in B(0,M)$ the inequality 
        \begin{align}\label{lem:App: Linear Solution Operator Properties : Result 2}
            || \curlyL_{T}(\overline{f},\overline{g}) - \curlyL_{T}(\hat{f},\hat{g}) ||_{L^\infty(0,T;L^1(\TT))}
            \leq
            L_{\curlyL}(T,M) || (\overline{f},\overline{g}) - (\hat{f}, \hat{g}) ||_{L^\infty(0,T;L^1(\TT))},
        \end{align}
        where
        \begin{align*}
            L_{\curlyL}(T,M) := T C_0 \exp(2M + C_0).
        \end{align*}
    \end{lemma}
    \begin{proof}
        The proof follows completely the lines of Lemma 2 in \cite{Hill_Note} without any further adjustments, therefore we will not repeat it here.
    \end{proof}

As a next step, we compute the right hand sides of $\eqref{lem:App: Existence linearized BN : Result 1}$ via the nonlinear mapping
    \begin{align*}
        \curlyN \curlyL_{T}\colon L^\infty(0,T;L^\infty(\TT)) \to L^\infty(0,T;L^\infty(\TT)),
        \quad (\alpha,\rho) \mapsto (f,g),
    \end{align*}
    where
    \begin{align*}
        f := \Parti(\rho) - \pi, \quad g := \pi - \Parti(\rho).
    \end{align*}

    Notice that, since $(P,\gamma)$ is admissible (see condition \ref{condition_two_admissible_pressure} in Definition~\ref{Defi:Admissible Pressure Function}), there exist two constants $C_P \in (0,\infty)$ and $\beta \in [2,\infty)$, such that we have
    \begin{align}\label{Growth for P}
        P^\prime(r) \leq C_P + C_P r^{\beta -1}, \quad P(r) \leq C_P + C_P r^\beta \quad \forall \, r \in [0,\infty).
    \end{align}
    We will fix these constants for the rest of this section.
    We obtain the following properties for the mapping $\curlyN \curlyL_{T}$, which are exactly the ones from \cite{Hill_Note} adjusted to our pressure function:

    \begin{lemma}\label{lem:App: Nonlinear Solution Operator Properties}
        Under the assumptions of Lemma~\ref{thm:App:Existence BN Solutions}, there exists some $C_1>0$ only depending on
        \begin{align*}
            ||\pi||_{L^\infty(0,T_0;L^\infty(\TT))}, \gamma, C_P,
        \end{align*}
        such that for any $T\in(0,T_0]$, and any $R>0$, we have
        \begin{align}\label{lem:App: Nonlinear Solution Operator Properties : Result 1}
            \curlyN \curlyL_{T} ( B(0,R) ) \subseteq B(0, C_1(1 + R^2 + R^\beta) ),
        \end{align}
        where $B(0,R)$ denotes the ball of radius $R$ in $L^\infty(0,T;L^\infty(\TT))^2$.\\
        Moreover, we have for any $(\overline{\alpha},\overline{\rho}), (\hat{\alpha},\hat{\rho}) \in B(0,R)$ the inequality
        \begin{align}\label{lem:App: Nonlinear Solution Operator Properties : Result 2}
            || \curlyN \curlyL_{T} (\overline{\alpha}, \overline{\rho}) - \curlyN \curlyL_T (\hat{\alpha}, \hat{\rho}) ||_{L^\infty(0,T;L^1(\TT))}
            \leq
            L_{\curlyN \curlyL}(R) || (\overline{\alpha}, \overline{\rho}) - (\hat{\alpha}, \hat{\rho}) ||_{L^\infty(0,T;L^1(\TT))},
        \end{align}
        where
        \begin{align*}
            L_{\curlyN \curlyL}(R) := C_1( 1 + R +R^{\beta - 1}).
        \end{align*}
    \end{lemma}
    \begin{proof}
        For $(\alpha,\rho) \in B(0,R)$ and $ \curlyN \curlyL_{T}(\alpha,\rho) = (f,g)$, i.e.
        \begin{align*}
            f = \pi - \Parti(\rho), \quad g = \Parti(\rho) - \pi,
        \end{align*}
        we estimate using the relation $\eqref{Growth for P}$
        \begin{align*}
            ||f||_{L^\infty(0,T,L^\infty(\TT))}
            &\leq
            ||\pi||_{L^\infty(0,T_0;L^\infty(\TT))} + ||P(\rho)||_{L^\infty(0,T,L^\infty(\TT))} + \frac{\gamma}{2} ||\rho||^2_{L^\infty(0,T;L^\infty(\TT))}\\
            &\leq
            ||\pi||_{L^\infty(0,T_0;L^\infty(\TT))} + C_P + C_P||\rho||^\beta_{L^\infty(0,T;L^\infty(\TT))} \\
            &\qquad+ \frac{\gamma}{2} ||\rho||^2_{L^\infty(0,T;L^\infty(\TT))}\\
            &\leq \max\Bigl\{||\pi||_{L^\infty(0,T_0;L^\infty(\TT))}+C_P, \frac{\gamma}{2} \Bigr\}\bigl( 1 + R^2 + R^\beta\bigr).
        \end{align*}
        The estimate for $g$ is completely the same. 
        To verify assertion $\eqref{lem:App: Nonlinear Solution Operator Properties : Result 2}$, we fix $(\overline{\alpha},\overline{\rho}), (\hat{\alpha}, \hat{\rho}) \in B(0,R)$ and denote
        \begin{align*}
            \curlyN \curlyL_{T}(\overline{\alpha},\overline{\rho}) = (\overline{f},\overline{g}),
            \quad 
            \curlyN \curlyL_{T}(\hat{\alpha}, \hat{\rho}) = (\hat{f}, \hat{g}).
        \end{align*}
        Then we estimate for a.e. $t \in [0,T]$ using the mean value theorem and relation $\eqref{Growth for P}$
        \begin{align*}
            \TTint |\overline{f}(t,x) - \hat{f}(t,x)|\,\dd x
            &=
            \TTint |\Parti(\overline{\rho}(t,x)) - \Parti(\hat{\rho}(t,x))| \,\dd x\\
            &\leq 
            \max\limits_{\lambda \in [0,R]} \bigl\{|\Parti^\prime(\lambda)| \bigr\} \TTint |\overline{\rho}(t,x)-  \hat{\rho}(t,x) | \, \dd x\\
            &\leq
            \bigl(C_P + C_P R^{\beta-1} + \gamma R\bigr) \TTint |\overline{\rho}(t,x) - \hat{\rho}(t,x) | \, \dd x\\
            &\leq
            \max\bigl\{C_P,\gamma\bigr\}\bigl( 1 + R + R^{\beta-1} \bigr) \TTint |\overline{\rho}(t,x) - \hat{\rho}(t,x) |\, \dd x.
        \end{align*}
        The estimate for the difference between $\overline{g}$ and $\hat{g}$ is completely the same.
        Hence, setting 
        \begin{align*}
            C_1 := \max\bigl\{||\pi||_{L^\infty(0,T_0;L^\infty(\TT))}+C_p, \gamma\bigr\} 
        \end{align*}
        yields the claim.
    \end{proof}

    Finally, we provide the proof of Theorem~\ref{thm:App:Existence BN Solutions} via a fixed point argument.

    \begin{proof}[Proof of Theorem~\ref{thm:App:Existence BN Solutions}]
        For $T \in (0,T_0]$, we define the map
        \begin{align*}
            \Phi_{T}:= \curlyL_{T} \circ \curlyN \curlyL_{T}\colon 
            L^\infty(0,T,L^\infty(\TT))^2 \to L^\infty(0,T;L^\infty(\TT))^2.
        \end{align*}
        By definition of $\curlyL_{T}$ and $ \curlyN \curlyL_{T}$, any fixed point $(\alpha,\rho) \in L^\infty(0,T;L^\infty(\TT))^2$ of $\Phi_{T}$ satisfies $\alpha,\rho \in C([0,T];L^1(\TT))$ and $\eqref{thm:App:Existence BN Solutions : Result 1}, \eqref{thm:App:Existence BN Solutions : Result 3}$, so we are done with the first part of Theorem~\ref{thm:App:Existence BN Solutions}, if we show that there exists some $T_1 \in (0,T_0]$, such that the map $\Phi_{T_1}$ admits a fixed point. 
        To do this, we define
        \begin{align*}
            R_\star := 2 C_0, \quad T_\star:=\min\biggl(\frac{\ln(2)}{C_1(1 + (2C_0)^\beta + (2C_0)^2)}, 1, T_0\biggr),
        \end{align*}
        where $C_0$ and $C_1$ are the constants provided by Lemma~\ref{lem:App: Linear Solution Operator Properties} and Lemma~\ref{lem:App: Nonlinear Solution Operator Properties}, respectively.
        With $\eqref{lem:App: Nonlinear Solution Operator Properties : Result 1}$ and $\eqref{lem:App: Linear Solution Operator Properties : Result 1}$ we obtain
        \begin{align*}
            \curlyL_{T_\star} (\curlyN \curlyL_{T_\star} (B(0,R_\star)))
            \subseteq 
            B(0, C_0 \exp(T_\star C_1 (1 + R_\star^\beta + R_\star^2))).
        \end{align*}
        Since 
        \begin{align*}
            \exp(T_\star C_1 (1 + R_\star^\beta + R_\star^2))
            \leq
            \exp\bigl( \frac{C_1(1 + (2C_0)^\beta + (2C_0)^2)\ln(2)}{C_1(1+  (2C_0)^\beta + (2C_0)^2 )} \bigr)= 2,
        \end{align*}
        we have shown that
        \begin{align*}
            \Phi_{T_\star}(B(0,R_\star)) \subseteq B(0,R_\star),
        \end{align*}
        and in particular
        \begin{align}\label{thm:App: Existence BN Solutions : Proof 1}
            \Phi_{T}(B(0,R_\star)) \subseteq B(0,R_\star) \quad \forall \, T \in (0,T_\star].
        \end{align}
        By $\eqref{lem:App: Linear Solution Operator Properties : Result 2}$ and $\eqref{lem:App: Nonlinear Solution Operator Properties : Result 2}$, we have for any $(\overline{\alpha},\overline{\rho}), (\hat{\alpha}, \hat{\rho}) \in B(0,R_\star)$ and any $T \in (0,T_\star]$ that
        \begin{align*}
            &||\Phi_{T}(\overline{\alpha},\overline{\rho}) - \Phi_{T}(\hat{\alpha},\hat{\rho})||_{L^\infty(0,T,L^1(\TT))}\\
            &\leq
            L_{\curlyL}(T,C_1(1 + R_\star^2 + R_\star^\beta)) L_{\curlyN \curlyL}(R_\star) ||(\overline{\alpha},\overline{\rho}) - (\hat{\alpha},\hat{\rho})||_{L^\infty(0,T,L^1(\TT))}\\
            &=
            T C_0 \exp(2C_1(1 + R_\star^2 + R_\star^\beta) + C_0)C_1 (1 + R_\star^{2} + R_\star^{\beta-1}) ||(\overline{\alpha},\overline{\rho}) - (\hat{\alpha},\hat{\rho})||_{L^\infty(0,T;L^1(\TT))}.
        \end{align*}
        Thus, we can choose $T_1 \in (0,T_\star]$ so small, such that
        \begin{align}\label{thm:App: Existence BN Solutions : Proof 2}
            ||\Phi_{T_1}(\overline{\alpha},\overline{\rho}) - \Phi_{T_1}(\hat{\alpha},\hat{\rho})||_{L^\infty(0,T_1,L^1(\TT))}
            \leq
            \frac{1}{2}||(\overline{\alpha},\overline{\rho}) - (\hat{\alpha},\hat{\rho})||_{L^\infty(0,T_1,L^1(\TT))}.
        \end{align}
        We construct a sequence in $B(0,R_\star)$ via the recursion
        \begin{align*}
            (\alpha_{n+1},\rho_{n+1}) := \Phi_{T_1}(\alpha_n, \rho_n) \quad \forall \, n \in \NN_{0},
        \end{align*}
        where we consider for $n=0$ the initial data as functions in $B(0,R_\star)$ that are constant in time, so $(\alpha_0,\rho_0)(t):=(\alpha_0,\rho_0)$ for all $t \in [0,T_1]$. 
        By virtue of $\eqref{thm:App: Existence BN Solutions : Proof 1}$, we have
        \begin{align}\label{thm:App: Existence BN Solutions : Proof 3}
            (\alpha_n,\rho_n) \in B(0,R_\star) \quad \forall\, n \in \NN.
        \end{align}
        Due to $\eqref{thm:App: Existence BN Solutions : Proof 2}$, the sequence $(\alpha_n,\rho_n)_{n\in\NN}$ is a Cauchy sequence in $L^\infty(0,T_1,L^1(\TT))$ and therefore there exists some $(\alpha,\rho) \in L^\infty(0,T_1,L^1(\TT))$, such that
        \begin{align*}
            (\alpha_n,\rho_n) \longrightarrow (\alpha,\rho) \quad \text{in } L^\infty(0,T_1,L^1(\TT)).
        \end{align*}
        Relation $\eqref{thm:App: Existence BN Solutions : Proof 3}$ implies $(\alpha,\rho) \in L^\infty(0,T_1,L^\infty(\TT))$. 
        Finally we verify that $(\alpha,\rho)$ is a fixed point of the map $\Phi_{T_1}$. We estimate for any $n \in \NN$
        \begin{align*}
            &||\Phi_{T_1}(\alpha,\rho) - (\alpha,\rho)||_{L^\infty(0,T_1;L^1(\TT))}\\
            &\leq
            ||\Phi_{T_1}(\alpha,\rho) - \Phi_{T_1}(\alpha_n,\rho_n)||_{L^\infty(0,T_1;L^1(\TT))} 
            + 
            ||\Phi_{T_1}(\alpha_n,\rho_n) - (\alpha,\rho)||_{L^\infty(0,T_1;L^1(\TT))}\\
            &\leq
            \frac{1}{2} ||(\alpha,\rho) - (\alpha_{n},\rho_{n})||_{L^\infty(0,T_1;L^1(\TT))}
            +
            ||(\alpha_{n+1},\rho_{n+1}) - (\alpha,\rho)||_{L^\infty(0,T_1;L^1(\TT))}.
        \end{align*}
       For $n \to \infty$, the right-hand side of this inequality converges to zero and therefore
        \begin{align*}
            ||\Phi_{T_1}(\alpha,\rho) - (\alpha,\rho)||_{L^\infty(0,T_1;L^1(\TT))} = 0,
        \end{align*}
        which gives in particular $\Phi(\alpha,\rho) = (\alpha,\rho)$ a.e. in $(0,\tildeTT)\times \TT$, so that $(\alpha,\rho)$ is a fixed point of $\Phi_{T_1}$.
        To prove assertion $\eqref{thm:App:Existence BN Solutions : Result 2}$, we notice that the velocity field has the regularity $u \in L^1(0,T_0;W^{1,\infty}(\TT))$. 
        This provides us a Lipschitz continuous flow map
        \begin{align*}
            X\colon [0,T_0] \times [0,T_0] \times \TT \to \RR.
        \end{align*}
        Then, it is well-known (see for instance \cite{MANIGLIA2007601} in a more general context), that we can represent $\alpha$ and $\rho$ via
        \begin{align*}
            \alpha(t,x) = \alpha_0(X(0,t,x)) \exp\biggl(\int_0^t f(s,X(s,t,x)) \, \dd s\biggr)
        \end{align*}
        and
        \begin{align*}
            \rho(t,x) = \rho_0(X(0,t,x)) \exp\biggl(\int_0^t g(s,X(s,t,x)) \, \dd s\biggr)
        \end{align*}
        for almost all $(t,x) \in [0,T_1]\times \TT$, where
        \begin{align*}
            f := \Parti(\rho) - \pi, 
            \quad
            g:= -\partial_x u + \pi - \Parti(\rho).
        \end{align*}
        Since $\alpha_0\geq 0$ almost everywhere on $\TT$, we conclude $\alpha \geq 0$ almost everywhere on $[0,T_1] \times \TT$.
        Moreover, we estimate the density using the bound on $\rho_0$ as
        \begin{align*}
            M_0^{-1} \exp\biggl(\int_0^t -||g(s,\cdot)||_{L^\infty(\TT)}\, \dd s\biggr)
            \leq
            \rho(t,x)
            \leq
            M_0 \exp\biggl(\int_0^t ||g(s,\cdot)||_{L^\infty(\TT)} \,\dd s\biggr)
        \end{align*}
        for almost all $(t,x) \in [0,T_1]\times \TT$.
        Since $g \in L^1(0,T_1;L^\infty(\TT))$, we can restrict $T_1$ if necessary to obtain assertion $\eqref{thm:App:Existence BN Solutions : Result 2}$.
    \end{proof}

\section{A Uniqueness Result}\label{A Uniqueness Result}

The following uniqueness result is a variation of the uniqueness result provided in \cite{Hillairet_New_Physical}.
The proof follows a duality argument that can also be found in \cite{MANIGLIA2007601} for instance.
The precise formulation of the uniqueness result reads as follows:

\begin{theorem}\label{thm:App:Uniqueness}
    Let $T>0$. Let $\mathbf{V}=(V_1,V_2) \in L^1(0,T;C_b(\TT \times \RR))^2$ and $g \in L^1(0,T;C_b(\TT \times \RR))$ with
    \begin{align}
        \partial_1 V_1 \in L^1(0,T;L^\infty(\TT \times \RR)), 
        \quad 
        \partial_2 V_1 =0 \text{ a.e. in } (0,T)\times\TT\times\RR,\label{Uniqueness Eq2}\\
        \int_0^T \TTint \sup\limits_{x_2 \in \RR} |\partial_1 V_2(t,x_1,x_2)| \,\dd x_1 \,\dd t < \infty,\:
        \partial_2 V_2 \in L^1(0,T;L^\infty(\TT \times \RR)),\label{Uniqueness Eq3}\\
        \int_0^T \TTint \sup\limits_{x_2 \in \RR} |\partial_1 g(t,x_1,x_2)|\, \dd x_1\, \dd t <\infty,\:
        \partial_2 g \in L^1(0,T;L^\infty(\TT \times \RR))\label{Uniqueness Eq5}.
    \end{align}
    Assume that $\nu \in \Cw([0,T];\mathcal{M}(\TT \times \RR))$ satisfies
    \begin{align}\label{support of nu}
        \sup\limits_{t \in [0,T_0]} ||\nu(t)||_{\curlyM(\TT \times \RR)} < \infty, \quad \mathrm{spt}\, \bigl(\nu(t) \bigr) \subseteq \TT \times K \quad \forall\,  t \in [0,T],
    \end{align}
    for some compact subset $K \subseteq \RR$ and
    \begin{align}\label{thm:App:Uniqueness:Result 1}
        \partial_t \nu 
        +
        \mathrm{div}_{(x_1,x_2)}(\mathbf{V}\nu)
        +
        g \nu
        &= 0 
        \quad \text{in } \mathcal{D}^\prime((0,T)\times\TT\times\RR),
    \end{align}
    with $\nu(0) = 0$.\\
    Then we have $\nu \equiv 0$.
\end{theorem}
    \begin{proof}
    We proceed as in \cite{Hillairet_New_Physical} with a duality argument. 
    The measure $\nu$ satisfying $\eqref{thm:App:Uniqueness:Result 1}$, means that for any $\varphi \in C^\infty_c((0,T) \times \TT \times \RR)$, we have 
    \begin{align*}
        \int_0^T \langle \nu(t), \partial_t \varphi + u \cdot \nabla \varphi - g \varphi \rangle \,\dd t = 0.
    \end{align*}
    By a standard regularization argument, this implies that for any $t \in [0,T]$ and any $\varphi \in C^1_c([0,t]\times \TT \times \RR)$, we have
    \begin{align}\label{thm:App:Uniqueness:Proof 1}
        \langle \nu(t), \varphi(t,\cdot) \rangle 
        =
        \int_0^t 
        \langle 
        \nu(s),
        \partial_t \varphi + \mathbf{V} \cdot \nabla \varphi - g \varphi 
        \rangle
        \,\dd s
        = 0.
    \end{align}
    We take a standard mollifier on $\RR$, i.e. 
    \begin{align*}
        \omega \in C^\infty_c((0,1)), \quad
        \int_{\RR} \omega = 1, \quad
        0 \leq \omega \leq 1,\\
        \omega_\varepsilon(x) := \frac{1}{\varepsilon} \omega(\frac{x}{\varepsilon}) \quad \text{for } \varepsilon>0, \: x \in \RR,
    \end{align*}
    and mollify $u$ and $g$ via
    \begin{align*}
        \mathbf{V}^\varepsilon(t,x,\xi)
        =
        \int_\RR \int_{\RR} \omega_\varepsilon(x-y) \omega_\varepsilon(\xi-\eta) \mathbf{V}(t,y,\eta) \,\dd y\, \dd \eta,\\
        g^{\varepsilon}(t,x,\xi)
        =
        \RRint \RRint \omega_\varepsilon(x-y) \omega_\varepsilon(\xi - \eta) g(t,y,\eta) \, \dd y \, \dd \eta.
    \end{align*}
    Then we have for any $k \in \NN$ that
    \begin{align*}
        \mathbf{V}^\varepsilon,g^\varepsilon \in L^1(0,T;C^k_b(\TT \times \RR)).
    \end{align*}
    Using that $g \in L^1(0,T;C_b(\TT \times \RR))$, we have for almost all $t \in [0,T]$ the convergence 
    \begin{align*}
        ||g^\varepsilon(t) - g(t) ||_{L^\infty(\TT \times K)} \to 0 \quad \text{for } \varepsilon \to 0,
    \end{align*}
    Together with
    \begin{align}\label{inserted estimate}
        \int_0^T ||g_\varepsilon(t)||_{L^\infty(\TT \times \RR)}\, \dd t 
        \leq
        \int_0^T ||g(t)||_{L^\infty(\TT \times \RR)} \, \dd t,
    \end{align}
    this implies using Lebesgue's dominated convergence theorem
    \begin{align}\label{thm:App:Uniqueness:Proof 3}
        ||g^\varepsilon - g||_{L^1(0,T;L^\infty(\TT \times K))} \to 0 \quad \text{for } \varepsilon \to 0.
    \end{align}
    By the same reasoning, we obtain
    \begin{align}\label{thm:App:Uniqueness:Proof 4}
        ||\mathbf{V}^\varepsilon - \mathbf{V}||_{L^1(0,T;L^\infty(\TT \times K))} \to 0 \quad \text{for } \varepsilon \to 0.
    \end{align}
    From now on, we denote by $\mathcal{C}$ a generic constant that may vary from line to line but does not depend on $\varepsilon$.
    By virtue of $\nabla V_1 \in L^1(0,T;L^\infty(\TT \times \RR))$, we obtain for almost all $t \in [0,T]$ and any $x,\xi \in \RR$
    \begin{align*}
        |V_1^\varepsilon(t,x,\xi) - V_1(t,x,\xi)| 
        &\leq
        \int_{B_\varepsilon(x)} \int_{B_\varepsilon(\xi)} \omega_\varepsilon(x-y) \omega_\varepsilon(\xi-\eta) 
        |V_1(t,y,\eta) - V_1(t,x,\xi)|\, \dd y \,\dd \eta\\
        &\leq
        \sqrt{2}\varepsilon||\nabla V_1(t)||_{L^\infty(\TT \times \RR)} \int_{B_\varepsilon(x)} \int_{B_\varepsilon(\xi)} \omega_\varepsilon(x-y) \omega_\varepsilon(\xi-\eta)\,  \dd y\,  \dd \eta\\
        &=
        \sqrt{2} ||\nabla V_1(t)||_{L^\infty(\TT \times \RR)} \varepsilon.
    \end{align*}
    Thus,
    \begin{align}\label{thm:App:Uniqueness:Proof 5}
        ||V_1^\varepsilon - V_1||_{L^1(0,T;L^\infty(\TT \times \RR)} 
        \leq \mathcal{C} \varepsilon.
    \end{align}
    We calculate for almost all $t \in [0,T]$ and any $x,\xi \in \TT \times \RR$:
    \begin{align*}
        \partial_2 g^\varepsilon(t,x,\xi)
        &=
        \RRint \RRint \omega_\varepsilon(x-y) \omega_\varepsilon(\xi - \eta) \partial_2 g(t,y,\eta)\, \dd y \, \dd \eta\\
        &\leq
        ||\partial_2g(t)||_{L^\infty(\TT \times \RR)} \RRint \RRint \omega_\varepsilon(x-y) \omega_\varepsilon(\xi - \eta) \, \dd y \, \dd \eta \\
        &= 
        ||\partial_2 g||_{L^\infty(\TT \times \RR)},
    \end{align*}
    so that, using $\partial_2 g \in L^1(0,T;L^\infty(\TT \times \RR))$,
    \begin{align}\label{thm:App:Uniqueness:Proof 6}
        ||\partial_2 g^\varepsilon||_{L^1(0,T;L^\infty(\TT \times \RR))}
        \leq \mathcal{C}.
    \end{align}
    By the same reasoning, we observe
    \begin{align}\label{thm:App:Uniqueness:Proof 7}
        ||\partial_2V_2^\varepsilon||_{L^1(0,T;L^\infty(\TT\times \RR))} + ||\nabla V_1^\varepsilon||_{L^1(0,T;L^\infty(\TT \times \RR)} \leq \mathcal{C}.
    \end{align}
    Finally, using $\eqref{Uniqueness Eq2}$ and $\eqref{Uniqueness Eq5}$, we observe that
    \begin{align}\label{thm:App:Uniqueness:Proof 8}
        ||\partial_1 V_2^\varepsilon||_{L^1(0,T;L^\infty(\TT \times \RR))}
        +
        ||\partial_1 g||_{L^1(0,T;L^\infty(\TT \times \RR))}
        \leq
        \frac{\mathcal{C}}{\sqrt{\varepsilon}}.
    \end{align}
    Now, we fix some $\varphi^\sharp \in C^\infty_c(\TT \times \RR)$ and consider for $t \in [0,T]$ the backward convection problem
    \begin{alignat}{2}\label{thm:App:Uniqueness:Proof 9}
    \left\{
        \begin{aligned}
            \partial_t \varphi^\varepsilon + \mathbf{V}^\varepsilon \cdot \nabla \varphi^\varepsilon &= g^\varepsilon \varphi^\varepsilon
            \quad &&\text{in } (0,t) \times \TT \times \RR,\\
            \varphi^\varepsilon(t,\cdot) &= \varphi^\sharp \quad &&\text{in } \TT \times \RR.
        \end{aligned}
    \right.
    \end{alignat}
    By well-known results on advection equations, $\eqref{thm:App:Uniqueness:Proof 9}$ admits a solution $\varphi^\varepsilon$ with regularity
    \begin{align*}
        \varphi^\varepsilon \in W^{1,1}([0,t];C^k_b(\TT \times \RR)) \quad \forall k \in \NN_{k\geq 1}.
    \end{align*}
    Moreover, since $\varphi^\sharp$ has compact support, there exists some compact set $L \subseteq \RR$ such that
    \begin{align*}
        \mathrm{spt} \, \varphi^\varepsilon(s,\cdot,\cdot) \subseteq L \quad \forall \, s \in [0,t].
    \end{align*}
    By a standard regularization argument via time-mollification, we are allowed to use $\varphi^\varepsilon$ as a test function in $\eqref{thm:App:Uniqueness:Proof 1}$.
    We obtain
    \begin{align}\label{thm:App:Uniqueness:Proof 10}
        \langle \nu(t), \varphi^\sharp \rangle
        &= 
        \int_0^t \langle \nu(s), 
        \partial_t \varphi^\varepsilon 
        +
        \mathbf{V} \cdot \nabla \varphi^\varepsilon
        -
        g \varphi^\varepsilon 
        \rangle \, \dd s \nonumber\\
        &=
        \int_0^t \langle \nu(s), (\mathbf{V}-\mathbf{V}^\varepsilon) \cdot \nabla \varphi^\varepsilon - (g - g^\varepsilon) \varphi^\varepsilon \rangle \, \dd s,
    \end{align}
    where we have used $\eqref{thm:App:Uniqueness:Proof 9}$ in the last equality.
    Using $\eqref{support of nu}$, we estimate
    \begin{align*}
        \langle \nu(t), \varphi^\sharp \rangle 
        \leq
        \mathcal{C}( I_1 + I_2 + I_3)
    \end{align*}
    with
    \begin{align*}
        I_1 &= \int_0^t ||(V_1 - V_1^\varepsilon) \partial_1 \varphi^\varepsilon ||_{L^\infty(\TT \times K)}\, \dd s,\\
        I_2 &= \int_0^t ||(V_2 - V_2^\varepsilon) \partial_2 \varphi^\varepsilon ||_{L^\infty(\TT \times K)}\,\dd s,\\
        I_3 &= \int_0^t || (g - g^\varepsilon) \varphi^\varepsilon ||_{L^\infty(\TT \times K)}\, \dd s.
    \end{align*}
    Applying the classical maximum principle for the advection equation $\eqref{thm:App:Uniqueness:Proof 9}$ yields in combination with $\eqref{inserted estimate}$ for any $s \in [0,t]$
    \begin{align}\label{thm:App:Uniqueness:Proof 11}
        ||\varphi^\varepsilon(s)||_{L^\infty(\TT \times \RR)}
        \leq
        ||\varphi^\sharp||_{L^\infty(\TT \times \RR)} 
        \exp\biggl( \int_0^t ||g^\varepsilon(s)||_{L^\infty(\TT \times \RR)} \, \dd s\biggr)
        \leq \mathcal{C},
    \end{align}
    so that with the convergence $\eqref{thm:App:Uniqueness:Proof 3}$
    \begin{align*}
        |I_3| \leq C \int_0^t ||g(s) - g^\varepsilon(s)||_{L^\infty(\TT \times K)} \dd s \to 0
        \quad \text{for } \varepsilon \to 0.
    \end{align*}
    For $I_1,I_2$, we need estimates for $\varphi_1:= \partial_1 \varphi^\varepsilon, \varphi_2:=\partial_2 \varphi^\varepsilon$. Differentiating $\eqref{thm:App:Uniqueness:Proof 9}$ with respect to the second spatial variable, we notice that $\varphi_2^\varepsilon$ satisfies the following advection problem
    \begin{alignat*}{2}
    \left\{
        \begin{aligned}
            \partial_t \varphi_2^\varepsilon 
            + 
            \mathbf{V}^\varepsilon \cdot \nabla \varphi_2^\varepsilon 
            &=
            \partial_2 g^\varepsilon \varphi^\varepsilon
            +
            (g^\varepsilon
            -
            \partial_2 V_2^\varepsilon )\varphi_2^\varepsilon
            \quad &&\text{in } (0,t)\times \TT \times \RR,\\
            \varphi_2^\varepsilon(t) &= \partial_2 \varphi^\sharp \quad &&\text{in } \TT \times \RR.
        \end{aligned}
    \right.
    \end{alignat*}
    Applying again a classical maximum principle for this advection problem, we obtain for any $s \in [0,t]$
    \begin{align}\label{thm:App:Uniqueness:Proof:Inserted}
        ||\varphi_2^\varepsilon(s)||_{L^\infty(\TT \times \RR)}
        &\leq
        \biggl( ||\partial_2 \varphi^\sharp||_{L^\infty(\TT \times \RR)}
        +
        \int_0^t ||\partial_2 g^\varepsilon(\tau)||_{L^\infty(\TT \times \RR)} ||\varphi(\tau)||_{L^\infty(\TT \times \RR)}\,\dd \tau \biggr)\nonumber\\
        & \qquad\times \exp\biggl( \int_0^t  ||g^\varepsilon(\tau)||_{L^\infty(\TT \times \RR)} + ||\partial_2 V_2^\varepsilon(\tau)||_{L^\infty(\TT \times \RR)} \, \dd  \tau \biggr)\nonumber \\
        &\leq 
        \mathcal{C}.
    \end{align}
    Here, we have used $\eqref{inserted estimate}$, $\eqref{thm:App:Uniqueness:Proof 6}$, $\eqref{thm:App:Uniqueness:Proof 7}$ and $\eqref{thm:App:Uniqueness:Proof 11}$.
    From relation $\eqref{thm:App:Uniqueness:Proof:Inserted}$ we deduce with the convergence $\eqref{thm:App:Uniqueness:Proof 4}$
    \begin{align*}
        |I_2| \leq \mathcal{C} \int_0^t ||V_2(s) - V_2^\varepsilon(s)||_{L^\infty(\TT \times K)} \, \dd s
        \to 0 \quad \text{for } \varepsilon \to 0.
    \end{align*}
    Differentiating $\eqref{thm:App:Uniqueness:Proof 9}$ with respect to the first variable, we infer that $\varphi_1^\varepsilon$ satisfies the following advection problem:
    \begin{alignat*}{2}
    \left\{
        \begin{aligned}
            \partial_t \varphi_1^\varepsilon + u^\varepsilon \cdot \nabla \varphi_1^\varepsilon
            &=
            \partial_1 g^\varepsilon\varphi^\varepsilon 
            -
            \partial_1V_2^\varepsilon \varphi_2^\varepsilon
            +
            (g^\varepsilon - \partial_1 V_1^\varepsilon) \varphi_1^\varepsilon 
            \quad &&\text{in } (0,t) \times \TT \times \RR,\\
            \varphi_1(t) &= \partial_1 \varphi^\sharp \quad &&\text{in } \TT \times \RR.
        \end{aligned}
    \right.
    \end{alignat*}
    Applying again the maximum principle for this advection problem, we infer for any $s \in [0,t]$ that
    \begin{align*}
        ||\varphi_1(s)||_{L^\infty(\TT \times \RR)}
        &\leq
        \biggl(
        ||\partial_1 \varphi^\sharp||_{L^\infty(\TT \times \RR)} 
        + \int_0^t 
        ||\partial_1 g^\varepsilon \varphi^\varepsilon||_{L^\infty(\TT\times \RR)}
        + 
        ||\partial_1V_2^\varepsilon \varphi_2^\varepsilon||_{L^\infty(\TT \times \RR)} 
        \, \dd \tau 
        \biggr)\\
        & \qquad  \times 
        \exp\biggl(\int_0^t (
        ||g^\varepsilon||_{L^\infty(\TT \times \RR)}
        +
        ||\partial_1V_1^\varepsilon||_{L^\infty(\TT \times \RR)}
        ) \, \dd \tau \biggr)\\
        & \leq \frac{\mathcal{C}}{\sqrt{\varepsilon}},
    \end{align*}
    where we have used $\eqref{inserted estimate}$, $\eqref{thm:App:Uniqueness:Proof 8}$, $\eqref{thm:App:Uniqueness:Proof 11}$ and $\eqref{thm:App:Uniqueness:Proof:Inserted}$.
    Using $\eqref{thm:App:Uniqueness:Proof 5}$, we obtain finally
    \begin{align*}
        |I_1|
        \leq
        \frac{\mathcal{C}}{\sqrt{\varepsilon}} \int_0^t ||V_1(s) - V_1^\varepsilon(s)||_{L^\infty(\TT \times K)} \, \dd s 
        \leq
        \frac{C}{\sqrt{\varepsilon}} \varepsilon \to 0 
        \quad \text{for } \varepsilon \to 0.
    \end{align*}
    Overall, we have shown that $\eqref{thm:App:Uniqueness:Proof 10}$ yields in the limit $\varepsilon \to 0$:
    \begin{align*}
        \langle \nu(t), \varphi^\sharp \rangle = 0.
    \end{align*}
    Since $\varphi^\sharp \in C^\infty_c(\TT \times \RR)$ was arbitrary, we have shown
    \begin{align*}
        \langle \nu(t) ,\varphi \rangle = 0 \quad \forall \, \varphi \in C^\infty_c(\TT \times \RR).
    \end{align*}
    Since $C^\infty_c(\TT\times \RR)$ lies dense in $C^0_0(\TT \times \RR)$ with respect to $||\cdot||_{L^\infty(\TT \times \RR)}$, we have shown
    \begin{align*}
        \langle \nu(t), \varphi \rangle = 0 \quad \forall \, \varphi \in C^0_0(\TT \times \RR).
    \end{align*}
    which means exactly $\nu(t) =0$. Since $t \in [0,T]$ was arbitrary we conclude $\nu = 0$.
\end{proof}

\section*{Acknowledgements}
Funded by Deutsche Forschungsgemeinschaft (DFG, German Research Foundation) under Germany's Excellence Strategy - EXC 2075 - 390740016. We acknowledge the support by the Stuttgart Center for Simulation Science (SC SimTech).
C.~R.  acknowledges funding by the Deutsche Forschungsgemeinschaft (DFG, 
German Research Foundation) - SPP 2410 Hyperbolic Balance Laws in Fluid 
Mechanics: Complexity, Scales, Randomness (CoScaRa).

\bibliographystyle{plain}
\bibliography{references}

\end{document}